\documentclass[10pt,leqno]{amsart}
 \usepackage{graphicx}    
 \usepackage{amsmath}
\usepackage{amssymb}
\usepackage[T1]{fontenc}
\usepackage[latin1]{inputenc}
\usepackage[english]{babel}
\usepackage[arrow, matrix, curve]{xy}
\usepackage{amsthm}
\usepackage{amsmath,amscd}
\usepackage{mathrsfs}
\usepackage{epic,eepic}
\usepackage{dsfont}
\usepackage{bbm}
\numberwithin{equation}{section}

\usepackage{pst-plot, pstricks}
\usepackage{fancybox,amssymb,color}
\usepackage{hyperref,srcltx}
\usepackage[T1]{fontenc}
\usepackage[latin1]{inputenc}
\usepackage[english]{babel}

\DeclareMathOperator\GL{GL}
\DeclareMathOperator\id{id}
\DeclareMathOperator\Hom{Hom}

\DeclareMathOperator\gr{gr}

\DeclareMathOperator\Spec{Spec}
\DeclareMathOperator\Spf{Spf}
\DeclareMathOperator\Sp{Sp}

\DeclareMathOperator\Sym{Sym}

\DeclareMathOperator\Fil{Fil}

\DeclareMathOperator\Res{Res}

\DeclareMathOperator\WD{WD}

\renewcommand{\phi}{\varphi}

\newcommand{\dbl}{{\mathchoice{\mbox{\rm [\hspace{-0.15em}[}}
                              {\mbox{\rm [\hspace{-0.15em}[}}
                              {\mbox{\scriptsize\rm [\hspace{-0.15em}[}}
                              {\mbox{\tiny\rm [\hspace{-0.15em}[}}}}
\newcommand{\dbr}{{\mathchoice{\mbox{\rm ]\hspace{-0.15em}]}}
                              {\mbox{\rm ]\hspace{-0.15em}]}}
                              {\mbox{\scriptsize\rm ]\hspace{-0.15em}]}}
                              {\mbox{\tiny\rm ]\hspace{-0.15em}]}}}}
                              
\newcommand{\dpl}{{\mathchoice{\mbox{\rm (\hspace{-0.15em}(}}
                              {\mbox{\rm (\hspace{-0.15em}(}}
                              {\mbox{\scriptsize\rm (\hspace{-0.15em}(}}
                              {\mbox{\tiny\rm (\hspace{-0.15em}(}}}}
\newcommand{\dpr}{{\mathchoice{\mbox{\rm )\hspace{-0.15em})}}
                              {\mbox{\rm )\hspace{-0.15em})}}
                              {\mbox{\scriptsize\rm )\hspace{-0.15em})}}
                              {\mbox{\tiny\rm )\hspace{-0.15em})}}}}

\newcommand{\Acal}{\mathcal{A}}
\newcommand{\Bcal}{\mathscr{B}}
\newcommand{\Ccal}{\mathcal{C}}
\newcommand{\Dcal}{\mathcal{D}}
\newcommand{\Ecal}{\mathcal{E}}

\newcommand{\Gcal}{\mathcal{G}}
\newcommand{\Hcal}{\mathcal{H}}
\newcommand{\Mcal}{\mathcal{M}}

\newcommand{\Ocal}{\mathcal{O}}
\newcommand{\Rcal}{\mathcal{R}}
\newcommand{\Scal}{\mathcal{S}}
\newcommand{\Tcal}{\mathcal{T}}
\newcommand{\Ttil}{\tilde\Tcal}

\newcommand{\Vcal}{\mathcal{V}}
\newcommand{\Wcal}{\mathcal{W}}

\newcommand{\Zcal}{\mathcal{Z}}

\newcommand{\Q}{\mathbb{Q}}
\newcommand{\R}{\mathbb{R}}
\newcommand{\Z}{\mathbb{Z}}
\newcommand{\Cbb}{\mathbb{C}}
\newcommand{\C}{\mathbb{C}}

\newcommand{\Abb}{\mathbb{A}}

\newcommand{\Fbb}{\mathbb{F}}
\newcommand{\Gbb}{\mathbb{G}}
\newcommand{\Pbb}{\mathbb{P}}
\newcommand{\Ubb}{\mathbb{U}}

\newcommand{\Mfrak}{\mathfrak{M}}
\newcommand{\Nfrak}{\mathfrak{N}}

\newcommand{\Xfrak}{\mathfrak{X}}

\newcommand{\mfrak}{\mathfrak{m}}

\newtheorem{theo}{Theorem}[section]
\newtheorem{lem}[theo]{Lemma}
\newtheorem{prop}[theo]{Proposition}
\newtheorem{cor}[theo]{Corollary}

\theoremstyle{remark}
\newtheorem{rem}[theo]{Remark}
\theoremstyle{remark}

\theoremstyle{definition}
\newtheorem{defn}[theo]{Definition}
\newtheorem{notation}[theo]{Notation}

\usepackage{color}
\newcounter{commentcounter}

\newcounter{commentcounterS}

\def\?{\ 
{\bf\color{red}???}\ 
\immediate\write16{}
\immediate\write16{Warning: There was still a question mark . . . }
\immediate\write16{}}

\begin{document}

\title[Density of potentially crystalline representations]{Density of potentially crystalline representations of fixed weight}
\author[E. Hellmann, B. Schraen]{Eugen Hellmann and Benjamin Schraen}
\begin{abstract}
Let $K$ be a finite extension of $\Q_p$ and let $\bar\rho$ be a continuous, absolutely irreducible representation of its absolute Galois group with values in a finite field of characteristic $p$. We prove that the Galois representations that become crystalline of a fixed regular weight after an abelian extension are Zariski-dense in the generic fiber of the universal deformation ring of $\bar\rho$. In fact we deduce this from a similar density result for the space of trianguline representations. This uses an embedding of eigenvarieties for unitary groups into the spaces of trianguline representations as well as the corresponding density claim for eigenvarieties as a global input.
\end{abstract}
\maketitle

\section{introduction}

The density of crystalline representations in the generic fiber of a local deformation ring plays an important role in the $p$-adic local Langlands correspondence for $\GL_2(\Q_p)$ and was proven by Colmez \cite{Colmez} and Kisin \cite{Kisin} for $2$-dimensional representations of ${\rm Gal}(\bar\Q_p/\Q_p)$. This density statement was generalized by Nakamura \cite{Nakamura2} and Chenevier \cite{Chtrianguline} to the case of $2$-dimensional representations of ${\rm Gal}(\bar\Q_p/K)$ for finite extensions $K$ of $\Q_p$ resp.~to the case of $d$-dimensional representations of ${\rm Gal}(\bar\Q_p/\Q_p)$ and finally the general case was treated in \cite{Nakamura3}.

In this paper we prove a slightly different density result in the generic fiber of a local deformation ring. 
The above density statements make heavy use of the fact that the Hodge-Tate weights of the crystalline representations may vary arbitrarily. Contrary to this case, we fix the Hodge-Tate weights but vary the level, or, more precisely, we allow finite (abelian) ramification and allow the representation to be potentially crystalline (more precisely crystabelline).

Note that this density statement is of a different nature than the density of crystalline representations. 
The density of crystalline representations holds true in the rigid generic fiber $(\Spf R_{\bar\rho})^{\rm rig}$ of the universal deformation ring $R_{\bar\rho}$ of a given residual ${\rm Gal}(\bar\Q_p/K)$-representation $\bar\rho$. 
In contrast to this result, the density of potentially crystalline representations of fixed weight only holds true in the "naive" generic fiber $\Spec(R_{\bar\rho}[1/p])$, as the set of representations with fixed (generalized) Hodge-Tate weights is Zariski-closed in the rigid generic fiber $(\Spf R_{\bar\rho})^{\rm rig}$. 

In the special case of $2$-dimensional potentially Barsotti-Tate representations of ${\rm Gal}(\bar\Q_p/\Q_p)$ our result gives a positive answer to a question of Colmez \cite{Colmez}.

A proof of this result (for $2$-dimensional representations of ${\rm Gal}(\bar\Q_p/\Q_p)$), using the $p$-adic local Langlands correspondence, was announced previously by Emerton and Pa\v{s}k\=unas. Our approach does not make use of such a correspondence and works in all dimensions and for arbitrary finite extensions of $\Q_p$. 
We were motivated by the case of $2$-dimensional potentially Barsotti-Tate representations, as for these representations an automorphy lifting theorem is known \cite{Kisinmodularity}. We hope to apply our density result to patching techniques in the future. 

More precisely our results are as follows.
Let $K$ be a finite extension of $\Q_p$ and let $\Gcal_K={\rm Gal}(\bar\Q_p/K)$ denote its absolute Galois group. Fix an absolutely irreducible continuous representation $\bar\rho:\Gcal_K\rightarrow \GL_d(\Fbb)$ with values in a finite extension $\Fbb$ of $\Fbb_p$. As the representation is assumed to be absolutely irreducible the universal deformation ring $R_{\bar\rho}$ of $\bar\rho$ exists. 

\begin{theo}\label{Thm1}
Assume that $p\nmid 2d$ and $\bar\rho\not\cong \bar\rho(1)$. Let ${\bf k}=(k_{i,\sigma})\in \prod_{\sigma:K\hookrightarrow \bar\Q_p}\Z^d$ be a regular weight. Then the representations that are crystabelline of labeled Hodge-Tate weight ${\bf k}$ are Zariski-dense in $\Spec R_{\bar\rho}[1/p]$.
\end{theo}
Similarly to the proof of density of crystalline representations we use a so called \emph{space of trianguline representations}. This space should be seen as a local Galois-theoretic counterpart of an eigenvariety of Iwahori level. Indeed it was shown in \cite{finslope} that certain eigenvarieties embed into a space of trianguline representations in the case $K=\Q_p$. This result is generalized to the case of an arbitrary extension $K$ of $\Q_p$ in section $\ref{eigenvar}$ below.  
In fact we prove the following density result for eigenvarieties which might be of independent interest. 

Let $E$ be an imaginary quadratic extension of a totally real field $F$ such that $[F:\Q]$ is even and let $G$ be a definite unitary group over $F$ which is quasi-split at all finite places. Let $Y$ be an eigenvariety for a certain set of automorphic representations of $G(\Abb_F)$ as in \cite[3]{Ch3} which comes along with a Galois pseudo-character interpolating the Galois representations attached to the automorphic representations at the classical points of $Y$. Given an absolutely irreducible residual representation $\bar\rho:{\rm Gal}(\bar\Q/E)\rightarrow \GL_d(\Fbb)$ there is an open and closed subspace $Y_{\bar\rho}\subset Y$ where the pseudo-character reduces to (the pseudo-character attached to) $\bar\rho$ modulo $p$. This gives rise to a map $Y_{\bar\rho}\rightarrow (\Spf R_{\bar\rho})^{\rm rig}$ to the rigid generic fiber of the universal deformation ring $R_{\bar\rho}$ of $\bar\rho$. 
\begin{theo}\label{Thm2}
Fix an algebraic irreducible representation $W$ of $G(F\otimes_{\Q}\R)$.
Let $f\in R_{\bar\rho}$ such that $f$ vanishes on all classical points $z\in Y_{\bar\rho}$ corresponding to irreducible automorphic representations $\Pi$ with $\Pi_{\infty}=W$. Then $f$ vanishes in $\Gamma(Y_{\bar\rho},\Ocal_Y)$.
\end{theo}

We prove Theorem $\ref{Thm1}$ by extending Theorem $\ref{Thm2}$ to the space of trianguline representations $X(\bar\rho_{w_0})$, using a map $f:Y_{\bar\rho}\rightarrow X(\bar\rho_{w_0})$ constructed in Theorem $\ref{maptofinslopespace}$ below. Here $\bar\rho_{w_0}$ is the restriction of $\bar\rho$ to the decomposition group at some place $w_0$ of $E$ dividing $p$.  
The second step in the proof of Theorem $\ref{Thm1}$ then is to realize a given residual representation $\bar\rho:\Gcal_K\rightarrow \GL_d(\Fbb)$ as the restriction to the decomposition group at $w_0$ of a ${\rm Gal}(\bar\Q/E)$-representation arising from an automorphic representation of $G(\Abb_F)$. This construction was already carried out in \cite{KisinGee} or \cite{EmertonGee} for example. Using Theorem \ref{Thm2}, we are then able to prove that an element of $R_{\bar\rho_{w_0}}$ vanishing at all crystabelline points of Hodge-Tate weight ${\bf k}$ has to vanish on all those irreducible components of $X(\bar\rho_{w_0})$ that contain (the image of) an eigenvariety.

The final step is to use use the Zariski-density of the image of the space of trianguline representations in the deformation space, which is the main result of \cite{Nakamura2} and \cite{Chtrianguline}. We have however to refine this density statement replacing this space of trianguline representations by the union of its irreducible components containing automorphic points of finite slope, {i.e.~we prove the following theorem (see the body of the paper for a more precise definition of $X(\bar\rho,W_\infty)^{\rm aut}$ and $Y(W_\infty,S,e)_{\bar\rho}$).
\begin{theo}
Let $X(\bar\rho,W_\infty)^{\rm aut}$ denote the union of those components of the space of trianguline representations that contain the image of some eigenvariety $Y(W_\infty,S,e)_{\bar\rho}$ of non-specified level $S$ away from $p$. Then $X(\bar\rho,W_\infty)^{\rm aut}$ has Zariski-dense image in the rigid analytic generic fiber of the universal deformation ring of $\bar\rho_{w_0}$.
\end{theo}} 
In this aim, we have to prove that this union of component contains sufficiently many crystalline points which are non critical and whose all refinements are non critical \emph{and} stay in this particular union of irreducible components. We then reduce this existence to the proof of the fact that such \emph{generic} crystalline points form a Zariski open subset of the \emph{scheme} parametrizing crystalline representations of fixed Hodge-Tate weights together with a density statement of automorphic points in an union of irreducible components of this space. The first of these two facts is proved using the existence of an universal Breuil-Kisin module on such a space and the second using the theory of Taylor-Wiles-Kisin systems.
\bigskip

{\bf Acknowledgments:} We would like to thank C.~Breuil, M.~Emerton, M.~Rapoport, P.~Scholze and V.~S\'echerre for helpful comments and conversations. 
Especially we would like to thank G.~Chenevier for answering a lot of questions and for pointing out a gap in the first version of the paper.
Further we thank O.~Taibi for pointing out a gap in the original argument.  
The first author was supported by a Forschungsstipendium He 6753/1-1 of the DFG and acknowledges the hospitality of the Institut de math\'ematiques de Jussieu in 2012 and 2013. Further he was partially supported by the SFB TR 45 of the DFG.
The second author was supported by the CNRS, the University of Versailles and the project Th\'eHopaD ANR-2011-BS01-005. He would like to thank IH\'ES for its hospitality during the last part of the redaction.
\bigskip

{\bf Notations:} We fix the following notations. Let $\bar\Q_p$\index{cl\^oture alg\'ebrique de $\Q_p$} be an algebraic closure of $\Q_p$. Let $K\index{base field}\subset \bar\Q_p$ be a finite extension of $\Q_p$ and let $K_0\index{unramified base field}$ denote the maximal unramified subextension of $\Q_p$ in $K$. We fix a compatible system $\epsilon_n\in\bar \Q_p$ of $p^n$-th roots of unity. Let $K_n=K(\epsilon_n)\subset \bar \Q_p$ and $K_\infty=\bigcup_n K_n$. We  will write $\Gcal_L={\rm Gal}(\bar \Q_p/L)$ for any subfield $L\subset \bar \Q_p$. Finally we write $\Gamma=\Gamma_K={\rm Gal}(K_\infty/K)$. We define the Hodge-Tate weights of a de Rham representation as the opposite of the gaps of the filtration on the covariant de Rham functor, so that the Hodge-Tate weight of the cyclotomic character is $+1$.

We choose a uniformizer $\varpi\in \Ocal_K$ and normalize the reciprocity isomorphism ${\rm rec}_K:K^\times\rightarrow W_K^{\rm ab}$ of local class field theory such that $\varpi$ is mapped to a geometric Frobenius automorphism. 
Here $W_K^{\rm ab}$ is the abelization of the Weil group $W_K\subset \Gcal_K$ and the reciprocity map allows us to identify $\Ocal_K^\times$ with a subgroup of $\Gcal_K^{\rm ab}$, the maximal abelian quotient of $\Gcal_K$.
Further we write $\varepsilon:\Gcal_K\rightarrow Z_p^\times$ for the cyclotomic character. 

Finally, given a crystalline  (resp.~semi-stable) representation $\rho:\Gcal_K\rightarrow \GL_d(\bar\Q_p)$ we write $D_{\rm cris}(\rho)$ (resp.~$D_{\rm st}(\rho)$) for the filtered $\varphi$-module (resp.~$(\varphi,N)$-module) associated to $\rho$. Further we write $\WD(D_{\rm cris}(\rho))$ (resp.~$\WD(D_{\rm st}(\rho))$) for the Weil-Deligne representation associated to $D_{\rm cris}(\rho)$ (resp.~$D_{\rm st}(\rho)$) by the recipe of Fontaine \cite{Fontaine}. A similar notation is used for potentially crystalline (resp.~potentially semi-stable representations).

\section{The space of trianguline representations}

Let $X$ be a rigid analytic space and recall the definition of the sheaf of relative Robba rings $\Rcal_X=\Rcal_{X,K}$ for $K$. If the base field $K$ is understood we will omit the subscript $K$ from the notation. This is the sheaf of functions that converge on the product of $X$ with some boundary part of the open unit disc over $K_0$, see \cite[2.2]{Galrep} or \cite[Definition 2.2.3]{KedlayaPX} for example\footnote{The sheaf $\Rcal_X$ is denoted by $\Bcal_{X, \rm rig}^\dagger$ in \cite{Galrep}}. If $X=\Sp L$ for a finite extension $L$ of $\Q_p$ we will write $\Rcal_L=\Rcal_{L,K}$ for (the global sections of) this sheaf.  
This sheaf of rings is endowed with a continuous $\Ocal_X$-linear ring homomorphism $\phi:\Rcal_X\rightarrow \Rcal_X$ and a continuous $\Ocal_X$-linear action of the group $\Gamma$. 
 Recall that a $(\phi,\Gamma)$-module over a rigid space $X$ consists of an $\Rcal_X$-module $D$ that is locally on $X$ finite free over $\Rcal_X$ together with a $\phi$-linear isomorphism $\Phi:D\rightarrow D$ and a semi-linear $\Gamma$-action commuting with $\Phi$. 

Let us write $\Ubb_L$ for the open unit disc over a $p$-adic field $L$ and $\Ubb_{r,L}\subset \Ubb_L$ for the admissible open subspace of points of absolute value $\geq r$ for some $r\in p^\Q\cap[0,1)$.
Given such an $r$ we write $\Rcal_X^r$ for the sheaf \[X\supset U\longmapsto \Gamma(U\times \Ubb_{r,K_0},\Ocal_{U\times \Ubb_{r,K_0}})\] and we write $\Rcal_X^+$ for the sheaf $\Rcal_X^0$ of functions converging on the product $X\times\Ubb_{K_0}$. 

Given a family of $\Gcal_K$-representations $\Vcal$ over a rigid space $X$, the work of Berger-Colmez \cite{BergerColmez} and Kedlaya-Liu \cite{KedlayaLiu} associates to $\Vcal$ a $(\phi,\Gamma)$-module ${\bf D}^\dagger_{\rm rig}(\Vcal)$ over $\Rcal_X$. 

Given a $(\phi,\Gamma)$-module $D$ over $X$, we write $H_{\phi,\Gamma}^\ast(D)$ for the cohomology of the complex
\[\begin{xy}
\xymatrix{
C_{\phi,\Gamma}^\bullet(D)=[D^{\Delta}\ar[rr]^{\phi-\id,\gamma-\id}&&D^{\Delta}\oplus D^{\Delta}\ar[rr]^{(\id-\gamma)\oplus(\phi-\id)}&&D^{\Delta}],
}
\end{xy}\] 
where $\Delta\subset \Gamma$ is the $p$-torsion subgroup of $\Gamma$ and $\gamma\in \Gamma/\Delta$ is a topological generator. It is known that the cohomology sheaves $H^i_{\phi,\Gamma}(D)$ are coherent $\Ocal_X$-modules for $i=0,1,2$ see \cite[Theorem 4.4.5]{KedlayaPX}.

\subsection{The parameters}
In this section, we recall the construction of the space $(\varphi,\Gamma)$-modules of rank $1$ over $\Rcal$ essentially following \cite{Colmez}. This is first step toward a construction of the {\emph{finite slope space}}.

Let $\Wcal=\Hom_{\rm cont}(\Ocal_K^\times,\Gbb_m(-))$ be the \emph{weight space} of $K$. This functor on the category of rigid analytic spaces is representable by the generic fiber of $\Spf \Z_p\dbl \Ocal_K^\times \dbr$. Further let $\Tcal=\Hom_{\rm cont}(K^\times,\Gbb_m(-))$. There is a natural projection $\Tcal\rightarrow \Wcal$ given by restriction to $\Ocal_K^\times$. The choice of the uniformizer $\varpi$ gives rise to a section of this projection and identifies $\Tcal$ with $\Gbb_m\times\Wcal$ via $\delta\mapsto (\delta(\varpi),\delta|_{\Ocal_K^\times})$. Especially $\Tcal$ is representable by a rigid space. 

We recall how the $(\phi,\Gamma)$-modules of rank $1$ over a rigid space $X$ are classified by $\Tcal(X)$, see \cite[Theorem 6.1.10]{KedlayaPX} (and also \cite[1.4]{Nakamura1} for the case $X=\Sp L$ in the context of $B$-pairs).

Let $X$ be a rigid space over $\Q_p$ and let $D$ be a rank $1$ family of $K$-filtered $\phi$-modules over $X$. Recall that this is a coherent $\Ocal_X\otimes_{\Q_p}K_0$-module that is locally on $X$ free of rank $1$ together with an $\id\otimes\phi$-linear automorphism $\Phi:D\rightarrow D$ and a filtration $\Fil^\bullet$ on $D_K=D\otimes_{K_0}K$ by $\Ocal_X\otimes_{\Q_p}K$ submodules that are locally on $X$ direct summands as $\Ocal_X$-modules. 

Assume that $X$ is defined over the normalization $K^{\rm norm}$ of $K$ inside $\bar\Q_p$ and assume that $D$ is free. Then such a $K$-filtered $\phi$-module may be described as follows.
There exists a uniquely determined $a\in \Gamma(X,\Ocal_X^\times)$ and uniquely determined $k_\sigma\in \Z$ for each embedding $\sigma:K\hookrightarrow K^{\rm norm}$ such that $D\cong D(a; (k_\sigma)_\sigma)$ where $\Phi^{[K_0:\Q_p]}$ acts on $D(a; (k_\sigma)_\sigma)$ via multiplication with $a\otimes \id\in \Gamma(X,\Ocal_X\otimes_{\Q_p}K_0)^\times$ and 
\begin{equation}\label{HTwts}
({\gr^i} D_K)\otimes_{\Ocal_X\otimes_{\Q_p}K,\id\otimes\sigma}\Ocal_X \cong \begin{cases}0 & i\neq {-k_\sigma} \\ \Ocal_X & i={-k_\sigma}\end{cases}
\end{equation}
for all embeddings $\sigma:K\hookrightarrow \bar\Q_p$.

Given $k_\sigma\in \Z$ for each embedding $\sigma:K\hookrightarrow \bar\Q_p$ we consider the following special $K$-filtered $\phi$-module $D((k_\sigma)_\sigma)$ over $L=K^{\rm norm}$ whose filtration is given by $(\ref{HTwts})$ and which has a basis on which {$\Phi^{[K_0:\mathbb{Q}_p]}$} acts via multiplication with $\prod_\sigma \sigma(\varpi)^{k_\sigma}$.

Let $X$ be a rigid space defined over $K^{\rm norm}$ and let $D$ be a $K$-filtered $\phi$-module over $X$. Associated to $D$ there is a $(\phi,\Gamma)$-module $\Rcal_X(D)$ of rank $1$ as follows. We write \[D=D(a; (0)_\sigma)\otimes_{K_0} D((k_\sigma)_\sigma)\] for some $k_\sigma\in \Z$ and $a\in \Gamma(X,\Ocal_X^\times)$ and define 
\[\Rcal_X(D(a;(0)_\sigma))= D(a;(0)_\sigma)\otimes_{\Ocal_X\otimes_{\Q_p}K_0}\Rcal_X,\]
where $\phi$ acts diagonally and $\Gamma$ acts by acting on the second factor.
 
Let us further fix a period $t_{\id}\in \Rcal^+_{K}$ of the Lubin-Tate character $\chi_{\rm LT}:\Gcal_K\rightarrow \Ocal_K^\times$ corresponding to our choice of a uniformizer $\varpi\in\Ocal_K$. Given $\sigma:K\hookrightarrow K^{\rm norm}$ we write $t_\sigma\in \Rcal^+_{K^{\rm norm}}$ for the period of $\sigma\circ\chi_{\rm LT}:\Gcal_K\rightarrow \Ocal_{K^{\rm norm}}^\times$. Then $t=\prod_\sigma t_\sigma$ is the usual period of the cyclotomic character $t=\log([(1,\epsilon_1,\epsilon_2,\dots)])\in\Rcal^+_{\Q_p}\subset \Rcal^+_{K^{\rm norm}}$.
Using this notations we write
\[\Rcal_X(D((k_\sigma)_\sigma))= \prod\nolimits_{\sigma} t_\sigma^{k_\sigma}\Rcal_X\subset \Rcal_X\big[\tfrac{1}{t}\big]\]
with action of $\phi$ and $\Gamma$ inherited from $\Rcal_X[1/t]$. 
Finally we set
\[\Rcal_X(D)=\Rcal_X(D(a;(0)_\sigma))\otimes_{\Rcal_X}\Rcal_X(D((k_\sigma)_\sigma)).\]
More generally, let $\delta:K^\times \rightarrow \Gamma(X,\Ocal_X^\times)$ be a continuous character. Then there is a $(\phi,\Gamma)$-module $\Rcal_X(\delta)$ of rank $1$ associated to $\delta$ as follows, cf.~\cite[Construction 6.1.4]{KedlayaPX}. Write $\delta=\delta_1\delta_2$ with $\delta_1|_{\Ocal_K^\times}=1$ and such that $\delta_2$ extends to a character of $\Gcal_K$. Then we set
\[\Rcal_X(\delta)=\Rcal_X(D(\delta_1(\varpi),(0)_\sigma))\otimes_{\Rcal_X}\mathbf{D}^\dagger_{\rm rig}(\delta_2).\]
We write $\delta(D)$ for the character of $K^\times$ such that $\Rcal_X(\delta(D))=\Rcal_X(D)$. 

Further, we can check that, given $k_\sigma\in \Z$, the character $\delta(D((k_\sigma)_\sigma))$ associated with the $K$-filtered $\varphi$-module $D((k_\sigma)_\sigma)$ is given by $\delta((k_\sigma)_\sigma)$,i.e.~by the character $z\mapsto \prod_\sigma \sigma(z)^{k_\sigma}$. We write $\delta_\Wcal((k_\sigma)_\sigma)$ for the restriction of $\delta((k_\sigma)_\sigma)$ to $\Ocal_K^\times$. Finally we {have $\varepsilon\circ{\rm rec}_K=\delta(1,\dots, 1)|\delta(1,\dots, 1)|$.}
\begin{lem}\label{rk1cohom}
Let $\delta\in \Tcal(L)$ for a local field $L\supset K^{\rm norm}$. Then
\begin{align*}
H_{\phi,\Gamma}^0(\Rcal(\delta))\neq 0&\Longleftrightarrow \delta=\delta((-k_\sigma)_\sigma)\ \text{for some}\ (k_\sigma)_\sigma\in \prod\nolimits_{\sigma:K\hookrightarrow L}\Z_{\geq 0},\\
H_{\phi,\Gamma}^2(\Rcal(\delta))\neq 0&\Longleftrightarrow \delta=\varepsilon\cdot\delta((k_\sigma)_\sigma)\ \text{for some}\  (k_\sigma)_\sigma\in \prod\nolimits_{\sigma:K\hookrightarrow L}\Z_{\geq 0}.
\end{align*}
Especially $H_{\phi,\Gamma}^1(\Rcal(\delta))$ has $L$-dimension $[K:\Q_p]$ if and only if
\[\delta\notin \left\{\delta((-k_\sigma)_\sigma), \varepsilon\cdot \delta((k_\sigma)_\sigma)\mid (k_\sigma)\in \prod\nolimits_\sigma \Z_{\geq 0}\right\}. \]
\end{lem}
\begin{proof} This is \cite[Proposition 6.2.8]{KedlayaPX}, cf.~also \cite[Proposition 2.14]{Nakamura1}.
\end{proof}

\begin{notation}
\noindent (i) Let us write $\Tcal_{\rm reg}\subset \Tcal$ for the set of \emph{regular characters}, i.e the characters 
\[\delta\notin \left\{\delta((-k_\sigma)_\sigma), \varepsilon\cdot \delta((k_\sigma)_\sigma)\mid (k_\sigma)\in \prod\nolimits_\sigma \Z_{\geq 0}\right\}\]
\noindent (ii) Let $d>0$ be an integer. We define the set of \emph{regular parameters} $\Tcal_{\rm reg}^d\subset \Tcal^d$ to be the set of $(\delta_1,\dots,\delta_d)\in \Tcal^d$ such that $\delta_i/\delta_j\in \Tcal_{\rm reg}$ for $i\leq j$. Note that by construction $\Tcal_{\rm reg}^d\neq (\Tcal_{\rm reg})^d$.

\noindent (iii) A weight $\delta\in \Wcal(\bar\Q_p)$ is \emph{algebraic of weight} $(k_\sigma)_\sigma$ if $\delta=\delta_{\Wcal}((k_\sigma)_\sigma)$. 

\noindent (iv) We say that $\delta\in \Wcal(\bar\Q_p)$ is \emph{locally algebraic} of weight $(k_\sigma)_\sigma$ if $\delta\otimes\delta_\Wcal((-k_\sigma)_\sigma)$ becomes trivial after restricting to some open subgroup of $\Ocal_K^\times$. 

\noindent (v) An element ${\bf k}=(k_{\sigma, i})_\sigma\in \prod_\sigma \Z^d$ is called \emph{strongly dominant} if $k_{\sigma, 1}>k_{\sigma,2}>\dots>k_{\sigma, d}$ for all $\sigma$. 

\noindent (vi) Let ${\bf k}=(k_{\sigma, i})_\sigma\in \prod_\sigma\Z^d$.  We say that $(\delta_1,\dots, \delta_d)\in \Wcal^d(\bar\Q_p)$ is \emph{algebraic of weight} ${\bf k}$ if $\delta_i$ is algebraic of weight $(k_{\sigma,i})_\sigma$.  An element $\delta=(\delta_1,\dots, \delta_d)\in \Wcal^d$ is called \emph{locally algebraic of weight ${\bf k}$} if $\delta_i$ is locally algebraic of weight $(k_{\sigma, i})_\sigma$. 
 The set of weights that are locally algebraic of weight ${\bf k}$ is denoted by $\Wcal^d_{{\bf k},{\rm la}}\subset \Wcal^d$. 
\end{notation}

\subsection{The space of trianguline $(\phi,\Gamma)$-modules}
We extend the construction of a space of trianguline $(\phi,\Gamma)$-modules with regular parameters given in \cite{Chtrianguline} to our context.

Let $d$ be a positive integer and consider the functor $\Scal_d^\square$ that assigns to a rigid space $X$ the isomorphism classes of quadruples $(D,\Fil_\bullet(D),\delta, \nu)$, where $D$ is a $(\phi,\Gamma)$-module over $\Rcal_X$ and $\Fil_\bullet(D)$ is a filtration of $D$ by sub-$\Rcal_X$-modules that are stable under the action of $\phi$ and $\Gamma$ and that are locally on $X$ direct summands as $\Rcal_X$-modules. Further $\delta\in \Tcal_{\rm reg}^d(X)$ and $\nu=(\nu_1,\dots, \nu_d)$ is a collection of trivializations
\[\nu_i:\Fil_{i+1}(D)/\Fil_i(D)\stackrel{\cong}{\longrightarrow} \Rcal_X(\delta_i).\]
Similarly, we consider a variant of this functor parametrizing non-split extensions, cf.~\cite{finslope}, that is, the functor $\Scal_d^{\rm ns}$ that assigns to $X$ the quadruples $(D,\Fil_\bullet(D),\delta,\nu_d)$ where $D$ and $\Fil_\bullet(D)$ are as above and $\delta\in \Tcal_{\rm reg}^d$ such that locally on $X$ there exist short exact sequences
\[0\longrightarrow \Fil_i(D)\longrightarrow \Fil_{i+1}(D)\longrightarrow \Rcal_X(\delta_i)\longrightarrow 0\]
that are non split at every geometric point $x\in X$ as a sequence of $(\phi,\Gamma)$-modules. Finally $\nu_d$ is a trivialization 
\[\nu_d:\Fil_{d+1}(D)/\Fil_d(D)\stackrel{\cong}{\longrightarrow} \Rcal_X(\delta_d).\]
\begin{prop}\label{H1locallyfree}
Let $\delta=(\delta_1,\dots, \delta_d)\in \big(\Tcal_{\rm reg}\big)^d(X)$ for some rigid space $X$ and let $D$ be a successive extension of the $\Rcal_X(\delta_i)$. Then $H_{\phi,\Gamma}^1(D)$ is a locally free $\Ocal_X$-module of rank $d[K:\Q_p]$. 
\end{prop}
\begin{proof}
It follows from \cite[Theorem 4.4.5]{KedlayaPX} that the cohomology is a coherent sheaf and it suffices to compute its rank at all points.
We proceed by induction on $d$. The rank $1$ case is settled by Lemma $\ref{rk1cohom}$. Consider the short exact sequence
\begin{equation}\label{extn}
0\longrightarrow \Rcal_X(\delta_1)\longrightarrow D\longrightarrow D'\longrightarrow 0.
\end{equation}
Then, by induction hypothesis, $H_{\phi,\Gamma}^1(D')$ is locally free of rank $(d-1)[K:\Q_p]$ and hence the Euler-characteristic formula \cite[Theorem 4.4.5 (2)]{KedlayaPX} implies that $H_{\phi,\Gamma}^0(D')=H_{\phi,\Gamma}^2(D')=0$. The claim now follows from the long exact sequence associated to $(\ref{extn})$ and the fact that $H_{\phi,\Gamma}^1(\Rcal_X(\delta_1))$ is locally free of rank $[K:\Q_p]$ and $H^2_{\phi,\Gamma}(\Rcal_X(\delta_1))=0$ by Lemma $\ref{rk1cohom}$. 
\end{proof}
\begin{theo}\label{constructionSd}
\noindent {\rm (i)} The functors $\Scal_d^\square$ and $\Scal_d^{\rm ns}$ are representable by rigid spaces. \\
\noindent {\rm (ii)} The map $\Scal_d^\square\rightarrow \Tcal_{\rm reg}^d$ is smooth of relative dimension $\tfrac{d(d-1)}{2}[K:\Q_p]$.\\
\noindent {\rm (iii)} The map $\Scal_d^{\rm ns}\rightarrow \Tcal_{\rm reg} ^d$ is smooth and proper and \[\dim\Scal_d^{\rm ns}=1+[K:\Q_p]\big(\tfrac{d(d+1)}{2}\big)\]
\end{theo}
\begin{proof}
The proof is the same as the proof of \cite[Theorem 3.3]{Chtrianguline} resp. \cite[Proposition 2.3]{finslope}. For the convenience of the reader we give a short sketch. The case $d=1$ is settled by $\Scal_1^\square=\Scal_1^{\rm ns}=\Tcal$. Now assume that $\Scal_{d-1}^\square$ and $\Scal_{d-1}^{\rm ns}$ are constructed with universal objects $\Dcal^\square_{d-1}$ resp. $\Dcal_{d-1}^{\rm ns}$. Let $U\subset \Tcal\times \Scal_{d-1}^\square$ resp. $V\subset \Scal_{d-1}^{\rm ns}\times \Tcal$ be the preimage of $\Tcal_{\rm reg}^d\subset \Tcal\times \Tcal_{\rm reg}^{d-1}$ under the canonical projection. Then Proposition $\ref{H1locallyfree}$ implies that 
\[\mathscr{E}xt^1_{\Rcal_U}(\Rcal_U(\delta_1),\Dcal_{d-1}^\square)=H^1_{\phi,\Gamma}(\Dcal_{d-1}^\square(\delta_1^{-1}))\]
resp.
\[\mathscr{E}xt^1_{\Rcal_V}(\Rcal_V(\delta_1),\Dcal_{d-1}^{\rm ns})=H^1_{\phi,\Gamma}(\Dcal_{d-1}^{\rm ns}(\delta_1^{-1}))\]
are vector bundles of rank $(d-1)[K:\Q_p]$. As the Tate-duality is a perfect pairing \cite[Theorem 4.4.5]{KedlayaPX} we find that also
\[\Mcal_U=\mathscr{E}xt^1_{\Rcal_U}(\Dcal_{d-1}^\square,\Rcal_U(\delta_1))\] 
resp. 
\[\Mcal_V=\mathscr{E}xt^1_{\Rcal_V}(\Dcal_{d-1}^{\rm ns},\Rcal_V(\delta_1))\] 
are vector bundles of rank $(d-1)[K:\Q_p]$. Now $\Scal_d^\square=\underline\Spec_U(\Sym^\bullet \Mcal_U^\vee)$ is the geometric vector bundle over $U$ associated to $\Mcal_U$ while $\Scal_d^{\rm ns}=\mathbb{P}_V(\Mcal^\vee_V)$ is the projective bundle associated to $\Mcal_V$. Here $\underline{\Spec}$ is the relative spectrum in the sense of \cite[2.2]{Conrad} and given a vector bundle $\Ecal$ the projective bundle $\Pbb(\Ecal)=\underline{\rm Proj}(\Sym^\bullet \Ecal)$ is the relative Proj in the sense of \cite[2.3]{Conrad}.

The universal object $\Dcal_d^\square$ then is the universal extension 
\[0\longrightarrow \Rcal(\delta_1)\longrightarrow \Dcal_d^\square\longrightarrow \Dcal_{d-1}^\square\longrightarrow 0\] over $\Scal_d^\square$. In the non-split context consider the geometric vector bundle $\tilde\Scal_d^{\rm ns}=\underline\Spec_V(\Sym^\bullet\Mcal_V^\vee)$ over $V$ associated to $\Mcal_V$. Then there is a universal extension
\[0\longrightarrow \Rcal(\delta_1)\longrightarrow \tilde\Dcal_d^{\rm ns}\longrightarrow \Dcal_{d-1}^{\rm ns}\longrightarrow 0\]  
over $\tilde\Scal_d^{\rm ns}$. Consider the open subspace $\tilde\Scal_d^{\rm ns}\backslash V\subset \tilde\Scal_d^{\rm ns}$ where the image of the zero section $0:V\hookrightarrow \tilde\Scal_d^{\rm ns}$ is removed. This space carries a natural action of $\Gbb_m$ and this action lifts to an action on the restriction of $\tilde\Dcal_d^{\rm ns}$ to $\tilde\Scal_d^{\rm ns}\backslash V$ by acting on $\Rcal(\delta_1)$. Hence $\tilde\Dcal_d^{\rm ns}$ descends to a $(\phi,\Gamma)$-module $\Dcal_d^{\rm ns}$ over $\Pbb(\Mcal^\vee_V)=(\tilde\Scal_d^{\rm ns}\backslash V)/\Gbb_m$.

The computation of the dimension follows from the construction as well as the fact that $\Scal_d^\square$ is smooth over $\Tcal_{\rm reg}^d$ and $\Scal_d^{\rm ns}$ is smooth and proper over $\Tcal_{\rm reg}^d$. 
\end{proof}
Let $r\in p^\Q\cap[0,1)$ and consider the ring $\Rcal^r=\Rcal_{\Q_p}^r$. If $n\gg 0$, then there is a morphism $\Rcal^r\rightarrow K_n\dbl t\dbr$ where the ring $K_n\dbl t\dbr$ is viewed as the complete local ring at the point of $\Ubb_{r, K_0}$ corresponding to (the ${\rm Gal}(\bar\Q_p/K_0)$-orbit of) $1-\epsilon_n$. If $D_r$ is a $(\phi,\Gamma)$-module defined over $\Rcal_L^r$ for some $p$-adic field $L$ and some $r\in p^\Q\cap[0,1)$ and if $D=D_r\otimes_{\Rcal_L^r}\Rcal_L$, then we define 
\begin{align*}
D_{\rm dR}(D)&=(K_\infty\otimes_{K_n} K_n\dpl t\dpr\otimes_{\Rcal_{\Q_p}^r}D_r)^\Gamma\\
\Fil^iD_{\rm dR}(D)&=(K_\infty\otimes_{K_n} t^iK_n\dbl t\dbr\otimes_{\Rcal_{\Q_p}^r}D_r)^\Gamma.
\end{align*}
If $L$ contains $K^{\rm norm}$, then $D_{\rm dR}(D)$ splits up into a product $D_{\rm dR}(D)=\prod_\sigma D_{{\rm dR},\sigma}(D)$ and $\Fil^iD_{\rm dR}(D)=\prod_\sigma \Fil_\sigma^i D_{\rm dR}(D)$ splits up into filtrations $\Fil^i_\sigma D_{\rm dR}(D)$ of the $D_{{\rm dR}, \sigma}(D)$.\\
As usual we can extend the notions of being crystalline or de Rham to $(\phi,\Gamma)$-modules.
\begin{defn} Let $L$ be a finite extension of $\Q_p$ and let $D$ be a $(\phi,\Gamma)$-module of rank $d$ over $\Rcal_L=\Rcal_{L,K}$. Assume that $D=D_r\otimes_{\Rcal_L^r}\Rcal_L$ for some $(\phi,\Gamma)$-module $D_r$ defined over $\Rcal_L^r$ and some $r<1$.\\ 
\noindent (i) The $(\phi,\Gamma)$-module $D$ is called \emph{de Rham} if $D_{\rm dR}(D)$ is a free $L\otimes_{\Q_p}K$-module of rank $d$.\\
\noindent (ii) The module $D$ is called \emph{crystalline} if $D_{\rm cris}(D)=D[1/t]^\Gamma$ is free of rank $d$ over $L\otimes_{\Q_p}K_0$. \\
\noindent (iii) The module $D$ is called \emph{crystabelline}  if $D\otimes_{\Rcal_{L,K}}\Rcal_{L,K'}$ is crystalline for some abelian extension $K'$ of $K$. 
\end{defn}
The following proposition is the generalization of \cite[Proposition 2.3.4]{BellaicheChenevier}\footnote{Note that Bellaiche and Chenevier use a different sign convention.} to our context and its proof is essentially the same as in the case $K=\Q_p$.
\begin{prop}\label{deRham}
Let $L$ be a finite extension of $\Q_p$ containing $K^{\rm norm}$ and let $D$ be a $(\phi,\Gamma)$-module of rank $d$ over $\Rcal_L$ that is a successive extension of rank $1$ objects $\Rcal_L(\delta_i)$. Assume that  $(\delta_1|_{\Ocal_K^\times},\dots, \delta_d|_{\Ocal_K^\times})$ is locally algebraic of weight ${\bf k}=(k_{\sigma, i})$  for some strongly dominant weight ${\bf k}$. Then $D$ is de Rham with labeled Hodge-Tate weights ${\bf k}$.
\end{prop}
\begin{proof}
Write $R=\bigcup_n (L\otimes_{\Q_p}K_n\dbl t\dbr)$ for the moment. We proceed by induction on $d$. The case $d=1$ easily follows {from} the fact that we may twist by characters $\delta$ such that $\delta|_{\Ocal_K^\times}=1$ and the fact that the claim is true for characters of $\Gcal_K^{\rm ab}=\hat\Z\times\Ocal_K^\times$ by the definition of locally algebraic weights. 

For simplicity we only treat the case $p\neq 2$. In this case the group $\Gamma$ is pro-cyclic. In the case $p=2$ one concludes similarly after taking invariants under the $2$-power torsion subgroup $\Delta$ of $\Gamma$. 

Let $\gamma\in \Gamma$ be a topological generator and let $\Gamma_0=\langle \gamma \rangle\subset \Gamma$. We will prove by induction on $1\leq j\leq d$ that for $(\prod_\sigma t^{k_{\sigma,j}}R\otimes_{\Rcal_{\Q_p}^r}\Fil_j(D)_r)^{\Gamma_0}\neq0$ for big enough $r$. Suppose we have the result for $j\leq d-1$. One deduces from the short exact sequence 
\begin{align*}
0&\rightarrow \prod\nolimits_\sigma\Fil_\sigma^{-k_{\sigma,d}}D_{\rm dR}(\Fil_{d-1}(D))\rightarrow \prod\nolimits_\sigma\Fil_\sigma^{-k_\sigma}D_{\rm dR}(D) \rightarrow \prod\nolimits_\sigma\Fil_\sigma^{-k_\sigma}\Rcal_L(\delta_d) \\ &\rightarrow H^1\big(\Gamma_0,\Fil_{d-1}(D)_r\otimes_{\Rcal_{\Q_p}^r}\big(\prod\nolimits_\sigma t_\sigma^{-k_{\sigma,d}}\big)R\big).
\end{align*}
that it suffices to show that 
\[H^1\big(\Gamma_0,\Fil_{d-1}(D)_r\otimes_{\Rcal_{\Q_p}^r}\big(\prod\nolimits_\sigma t_\sigma^{-k_{\sigma,d}}\big)R\big)=0.\]
To do so we are reduced to compute the first cohomology of $(\prod\nolimits_\sigma t_\sigma^{-k_\sigma})R\otimes_L R(\delta_i)$ for $i\leq d-1$. However, this cohomology vanishes, as $\prod_\sigma {t_\sigma}^{-k_{\sigma,d}}R\otimes\delta_j\simeq\prod_\sigma {t_\sigma}^{-k_{\sigma,d}+k_{\sigma,j}}R\otimes R(\delta)$ with $\delta$ a finite order character, and $-k_{\sigma,d}+k_{\sigma,j}>0$ for all $\sigma$ and hence
\[H^1\left(\Gamma_0, \big(\prod\nolimits_\sigma t_\sigma^{i_\sigma}\big)R\right)=\big(\prod\nolimits_\sigma t_\sigma^{i_\sigma}\big)R\left/ (\gamma-1)\big(\prod\nolimits_\sigma t_\sigma^{i_\sigma}\big)R\right.=0\]
if $i_\sigma>0$ for all embeddings $\sigma$. It follows that $D$ has to be de Rham.
\end{proof}
Let $\omega_d^{\square}:\Scal_d^\square\rightarrow \Wcal^d$ resp.~$\omega_d:\Scal_d^{\rm ns}\rightarrow \Wcal^d$ denote the projection to the weight space. 
\begin{cor}\label{Zopeninfib}
\noindent {\rm (i)} Let $w\in \Wcal_{\rm alg}^d$ be a strongly dominant algebraic weight. Then there is a non-empty Zariski-open subset $Z_{\rm cris}({w})\subset \omega_d^{-1}({w})$ such that all points of $Z_{\rm cris}({w})$ are crystalline $(\phi,\Gamma)$-modules.\\
\noindent {\rm (ii)} Let ${\bf k}\in \prod_\sigma \Z^d$ be strongly dominant and let $w\in \Wcal_{{{\bf k}}, \rm la}^d$ be a locally algebraic weight. Then there is a non-empty Zariski-open subset $Z_{{\rm pcris}}(w)\subset \omega_d^{-1}(w)$ 
such that all points of $Z_{\rm pcris}(w)$ are crystabelline.
\end{cor}
\begin{proof}
The proof is identical to the one of \cite[Theorem 3.14]{Chtrianguline}.\\
(i) As $w=(w_1,\dots, w_d)\in \Wcal^d$ is algebraic we may write $\Rcal(\delta_i)=\Rcal(D(\delta_i))$ for any character $\delta_i\in \Tcal$ restricting to $w_i$ on $\Ocal_K^\times$.
We write $D(\delta_i)=D(a_i, (k_\sigma)_\sigma)$ with $a_i=\delta_i(\sigma)\prod_\sigma \sigma(\varpi)^{{-k_{\sigma,i}}}$ and let 
\[Z_{\rm cris}(w)=\big\{(D,\Fil_\bullet(D),\delta,\nu_d)\in \omega_d^{-1}(w)\mid \tfrac{a_i}{a_j}\neq p^{\pm [K_0:\Q_p]}\ \text{for}\ i<j\big\}.\] 
Let $D$ be a $(\phi,\Gamma)$-module associated to some point in $Z_{\rm cris}(w)$ then $D$ is de Rham by Proposition $\ref{deRham}$ above and hence potentially semi-stable. As $D$ is a successive extension of crystalline $(\phi,\Gamma)$-modules it has to be semi-stable and we have to assure that the monodromy acts trivial. However the monodromy operator maps the $\Phi^f$-eigenspace with eigenvalue $\lambda$ to the $\Phi^f$-eigenspace with eigenvalue $p^f\lambda$, where $f=[K_0:\Q_p]$. As the possible eigenvalues of $\Phi^f$ are given by the $a_i$ the monodromy has to be trivial. 

\noindent (ii) Let {$w^{\rm sm}=w\cdot\delta({-{\bf k}})=(w_1,\dots, w_n)$} and let $K'$ be the abelian extension of $K$ corresponding to $\bigcap \ker w_i\subset \Ocal_K^\times\hookrightarrow \Gcal_K^{\rm ab}$.  Then the same argument as above yields a Zariski-open subset  $Z_{\rm pcris}(w)\subset \omega_d^{-1}(w)$ whose points are $(\phi,\Gamma)$-modules that become crystalline over $K'$. 
\end{proof}
\begin{rem}
In the case $d=2$ the second claim of the corollary above especially applies to the weight ${\bf k}=((0,1)_{\sigma})$,i.e.~to potentially Barsotti-Tate representations. 
If $d>2$ a corresponding statement for potentially Barsotti-Tate representations can not hold true any longer. There are no strongly dominant weights for potentially Barsotti-Tate representations in this case and the dimension of the flag variety parametrizing the Hodge-filtrations for weights that are not strongly dominant will be strictly smaller than the dimension of the space of extensions of $(\phi,\Gamma)$-modules.   
\end{rem}
\begin{lem}\label{nctrianggradedpieces}
Let $L\subset\bar\Q_p$ be a finite extension of the Galois closure $K^{\rm norm}$ of $K$ inside $\bar\Q_p$ and let $V$ be a crystalline representation of $\Gcal_K$  on a $d$-dimensional $L$-vector space with labeled Hodge-Tate weights ${\bf k}=(k_{\sigma,i})$ such that ${\bf k}$ is strongly dominant. Let $D=D_{\rm cris}(V)$ and assume that the $[K_0:\Q_p]$-th power of the crystalline Frobenius $\Phi_{\rm cris}$ on $\WD(D)=D\otimes_{L\otimes_{\Q_p}K_0}\bar\Q_p$ is semi-simple. Let $\lambda_1,\dots,\lambda_d$ be an ordering of its eigenvalues and assume that for all $\sigma$ one has
\begin{equation}\label{noncrit}
\begin{aligned}
\tfrac{[K:\Q_p]}{[K_0:\Q_p]}{\rm val}(\lambda_1)&< -k_{\sigma,2}-\sum\nolimits_{\sigma'\neq \sigma} k_{\sigma',1}\\
\tfrac{[K:\Q_p]}{[K_0:\Q_p]}{\rm val}(\lambda_1\dots\lambda_i)&< -k_{\sigma, i+1}-\sum\nolimits_{\sigma'\neq \sigma} k_{\sigma',i}- \sum\nolimits_{\sigma'}\sum\nolimits_{j=1}^{i-1}k_{\sigma',j} .
\end{aligned}
\end{equation}
Then there is a triangulation $0=D_0\subset D_{1}\subset \dots \subset D_d=D^\dagger_{\rm rig}(V)$ such that $D_i/D_{i-1}\cong D(\delta_i)$ with $\delta_i:K^\times\rightarrow L^\times$ given by
\begin{align*}
&\delta_i|_{\Ocal_K^\times}:z\longmapsto \prod\nolimits_\sigma \sigma(z)^{k_{\sigma,i}}\\
&\delta_i(\varpi)=\lambda_i \prod\nolimits_\sigma \sigma(\varpi)^{{k_{\sigma,i}}}.
\end{align*}
\end{lem}
\begin{proof}
Let $D_i\subset D^\dagger_{\rm rig}(V)$ be the filtration induced by a filtration $0=D'_0\subset D'_1\subset \dots\subset D'_d=D=D_{\rm cris}(V)$ by $\Phi_{\rm cris}$-stable subspaces such that the restriction of $\Phi_{\rm cris}^{[K_0:\Q_p]}$ to $\WD(D_i)$ has eigenvalues $\lambda_1,\dots, \lambda_i$. Then $D_i$ is stable under $\phi$ and $\Gamma$ and we need to compute the graded pieces.
However, the graded pieces are as claimed, if the filtration $D'_\bullet$ is in general position with all the Hodge filtrations $\Fil^\bullet_\sigma$ which is to say 
\[\big(D'_i\otimes_{K_0\otimes L,\sigma\otimes\id}\bar\Q_p\big)\oplus\big( \Fil^{-k_{\sigma,i+1}}D_K\otimes_{K\otimes L,\sigma\otimes\id}\bar\Q_p\big)=D\otimes_{K_0\otimes L,\sigma\otimes \id}\bar\Q_p=\WD(D).\] 
One easily sees that this is assured by weak admissibility and condition $(\ref{noncrit})$.
\end{proof}
\subsection{Construction of Galois-representations}\label{deffinslopespace}
Let $\bar\rho:\Gcal_K\rightarrow \GL_d(\Fbb)$ be an absolutely irreducible continuous representation, where $\Fbb$ is a finite field of characteristic $p$. Write $R_{\bar\rho}$ for the universal deformation ring of $\bar\rho$ and $\Xfrak_{\bar\rho}$ for the generic fiber of $\Spf R_{\bar\rho}$ in the sense of Berthelot. 

Let $X$ be a rigid space and let $T:\Gcal_K\rightarrow \Gamma(X,\Ocal_X)$ be a continuous pseudo-character of dimension $d$. We say that $T$ has \emph{residual type} $\bar\rho$ if for all $x\in X$ the semi-simple representation $\rho_x:\Gcal_K\rightarrow \GL_d(\Ocal_{\bar\Q_p})$ with ${\rm tr}\,\rho_x=(T\otimes k(x))\otimes_{k(x)}\bar\Q_p$ (which is uniquely determined up to conjugation) reduces to (the isomorphism class of) $\bar\rho$ modulo the maximal ideal of $\Ocal_{\bar\Q_p}$.

Then the rigid space $\Xfrak_{\bar\rho}$ represents the functor that assigns to a rigid space $X$ the pseudo-characters $T:\Gcal_K\rightarrow \Gamma(X,\Ocal_X)$ of dimension $d$ and residual type $\bar\rho$.
 
By \cite[Theorem 5.2]{Galrep} there exists a natural rigid space $\Scal_d^{\rm ns, adm}$ which is \'etale\footnote{In the set up of adic spaces the spaces $\Scal_d^{\rm ns, adm}$ is an open subspace of $\Scal_d^{\rm ns}$.} over $\Scal_d^{\rm ns}$ and a vector bundle $\Vcal$ on $\Scal_d^{\rm ns, adm}$ together with a continuous representation $\rho:\Gcal_K\rightarrow \GL(\Vcal)$ such that ${\bf D}^\dagger_{\rm rig}(\Vcal)$ is the restriction of the universal trianguline $(\phi,\Gamma)$-module.
In the set up of adic spaces cf.\cite{Huber} the space $\Scal_d^{\rm ns, adm}$ is an open subset of $\Scal_d^{\rm ns}$. In the following we will often use the point of view of adic space, as it is easier to deal with topological matters in this context.
That means we will embed the category of rigid spaces into the category of adic spaces as in \cite[1.1.11]{Huber}.

Let us write $\Scal(\bar\rho)\subset \Scal_d^{\rm ns, adm}$ for the open and closed subspace where the pseudo-character ${\rm tr}\rho$ has residual type $\bar\rho$.
Then we obtain a canonical map
\[\pi_{\bar\rho}:\Scal(\bar\rho)\longrightarrow \Xfrak_{\bar\rho}\times\Tcal_{\rm reg}^d.\]
\begin{defn}
Let $X(\bar\rho)$ be the Zariski-closure of ${\rm Im}(\pi_{\bar\rho})\subset \Xfrak_{\bar\rho}\times\Tcal_{\rm reg}^d$ in $\Xfrak_{\bar\rho}\times\Tcal^d$. 
The space $X(\bar\rho)$ is called the \emph{finite slope space} in the following.
\end{defn}



\begin{prop}\label{Xreglarge}
There exists a canonical Zariski open and dense subset $U$ of $X(\bar\rho)$ contained in ${\rm Im}(\pi_{\bar\rho})$ such that $\pi_{\bar\rho}^{-1}(U)\rightarrow U$ is an isomorphism. Moreover all Galois representations $\rho\in X(\bar\rho)$ are trianguline and one has
\[\dim X(\bar\rho)=1+[K:\Q_p]\tfrac{d(d+1)}{2}.\]
\end{prop}
\begin{proof}
Let us write $M=|\pi_{\bar\rho}(\Scal(\bar\rho))|$ for the underlying point set of the image of $\pi_{\bar\rho}$ and let $x=(\rho,\delta_1,\dots,\delta_d)\in M$. Then ${\bf D}^\dagger_{\rm rig}(\rho)$ is strictly trianguline with ordered parameters $\delta_1,\dots,\delta_d$ in the sense of {\cite[Definition 6.3.1]{KedlayaPX}}.
Let {$X_1$} denote the Zariski-closure of $M$ in $\Xfrak_{\bar\rho}\times \Tcal^d_{\rm reg}$ and $X_2=X(\bar\rho)$ denote its closure in $\Xfrak_{\bar\rho}\times\Tcal^d$. Further we write $\rho^{\rm un}$ for the pullback of the universal $\Gcal_K$-representation on $\Xfrak_{\bar\rho}$ to $X_2$ and $\delta_1,\dots,\delta_d$ for the pullback of the universal characters of $K^\times$ on $\Tcal^d$ to $X_2$.
 By {\cite[Corollary 6.3.10]{KedlayaPX}}, there exists a proper birational morphism $p:X'_2\rightarrow X_2$ such that there is a (unique) increasing filtration $\Fil_i$ on ${\bf D}^\dagger_{\rm rig}(p^\ast \rho^{\rm un})$ by $(\phi,\Gamma)$-submodules which is a strictly trianguline filtration with ordered parameters $p^\ast \delta_1,\dots,p^\ast\delta_d$ over a Zariski-open and dense subset $U'\subset X'_2$ containing $p^{-1}(M)$.
 Then $U$ is given by the intersection of the complement of $p(X'_2\backslash U')$ in $X_2$ with $X_1$ and the locus where all the extensions
 \begin{equation}\label{extensiononU}
 0\rightarrow\Fil_i {\bf D}^\dagger_{\rm rig}(p^\ast \rho^{\rm un})\rightarrow \Fil_{i+1} {\bf D}^\dagger_{\rm rig}(p^\ast \rho^{\rm un})\rightarrow \Rcal_L(\delta_i)\rightarrow0
 \end{equation}
 are non split. 
 Further, by construction, the map $\Scal^{\rm ns}(\bar\rho)\rightarrow X(\bar\rho)$ has a section $f$ over $U$ given by the filtration $(\ref{extensiononU})$.
 Let us write $V=\pi_{\bar\rho}^{-1}(U)$.
It follows that $f\circ\pi_{\bar\rho}={\rm id}_{U}$. On the other hand \cite[Proposition 3.10]{finslope} gives that $\pi_{\bar\rho}|_V\circ f$ equals the identity on $V$ on the level of rigid analytic points. As $V$ is separated it follows that $\pi_{\bar\rho}\circ f$ equals $\id_V$, as a set-theoretic map (or as a morphism of topological spaces). It remains to show that $f$ and $\pi_{\bar\rho}|_V$ identify the structure sheaves on $U$ and $V$. However, this follows from the existence of the section $f$ as a morphism of rigid spaces together with the fact that $V$ is reduced: the sections of the structure sheaf are identified with rigid analytic morphisms to $\Abb^1$ and for a reduced space such a morphisms is determined by its values on rigid analytic points. 
 Finally it follows that $\dim X(\bar\rho)=\dim U=\dim\Scal^{\rm ns}(\bar\rho)=\dim\Scal_d^{\rm ns}$.



The claim that all representations are trianguline follows from {\cite[Theorem 6.3.13]{KedlayaPX}.
}
\end{proof}

Although we do not know whether ${\rm Im}(\pi_{\bar\rho})$ is Zariski-open in $X(\bar\rho)$ (and hence has a structure as a rigid space), we can consider a canonically constructed Zariski open and dense subset of ${\rm Im}(\pi_{\bar\rho})$. We will now write $X(\bar\rho)^{\rm reg}$ for the open subset $U$ constructed in Proposition \ref{Xreglarge} and we will refer to this subset to the \emph{regular part of the finite slope space}. As it is isomorphic to an open subset of $\Scal^{\rm ns}(\bar\rho)$, it is smooth, so that its connected components are in bijection with irreducible components of $X(\bar\rho)$.
%
%

The following lemma is a direct consequence of the identification of $X(\bar\rho)^{\rm reg}$ with an open subset of $\Scal^{\rm ns}(\bar\rho)$ and the construction of $\Scal^{\rm ns}(\bar\rho)$ as an open subspace of a successive extension of vector bundles over $\Tcal^d_{\rm reg}$.
\begin{lem}\label{locallyaproduct}
Let $x\in X(\bar\rho)^{\rm reg}$ be a rigid analytic point. Then there exists a neighborhood of $x$ in $X(\bar\rho)^{\rm reg}$ that is isomorphic to the product of a neighborhood of $\omega_d(x)\in \Wcal^d$ with the closed unit disc of dimension $1+d(d-1)/2 \cdot [K:\Q_p]$.
\end{lem}

Similar to the space $\Scal^{\rm ns}(\bar\rho)$ we can define a subspace $\Scal^\square(\bar\rho)\subset \Scal^\square_d$ consisting of those trianguline $(\varphi,\Gamma)$-modules in $\Scal^\square$ that come from a $\Gcal_K$-representation whose associated pseudo-character has residual type $\bar\rho$. As in the discussion above we have a map $\pi_{\bar\rho}^\square:\Scal^\square(\bar\rho)\rightarrow \Xfrak_{\bar\rho}\times\Tcal^d$. 
\begin{lem}\label{nsvssquare}
The map $\pi_{\bar\rho}^\square:\Scal^\square(\bar\rho)\rightarrow \Xfrak_{\bar\rho}\times\Tcal^d$ factors over $X(\bar\rho)$.
\end{lem}
\begin{proof}
As $X(\bar\rho)$ is closed and $\Scal^\square(\bar\rho)$ is reduced it is enough to show that a dense subset of $\Scal^\square(\bar\rho)$ maps to $X(\bar\rho)$. However, the set $\tilde\Scal^\square_d$ of all points $x\in \Scal_d^\square$, where all the extensions
\[0\longrightarrow \Fil_i(D)\otimes k(x)\longrightarrow \Fil_{i+1}(D)\otimes k(x)\longrightarrow \Rcal_{k(x)}(\delta_i)\longrightarrow 0\]
are non-split is Zariski-open. It follows that $\tilde\Scal^\square_d$ meets every component of $\Scal^\square(\rho)$ and in hence fact the intersection $\Scal^\square(\bar\rho)\cap\tilde\Scal^\square_d$ is Zariski-open and dense in $\Scal^\square(\bar\rho)$. 
Now we conclude by remarking that there is a canonical map $\tilde\Scal_d^\square\rightarrow \Scal^{\rm ns}_d$ (which is in fact a $\Gbb_m^{d-1}$-torsor) that induces a map $q:\tilde\Scal^\square_d\cap \Scal^\square(\bar\rho)\rightarrow \Scal^{\rm ns}(\bar\rho)$ such that the map $\pi^\square_{\bar\rho}$ factors through $q$. 
\end{proof}

\section{Application of eigenvarieties}

In this section, we recall some facts about eigenvarieties attached to definite unitary groups and prove a density statement about them which will be used in the proof of the main theorem.
\subsection{The eigenvarieties}\label{eigenvarieties}
The eigenvarieties that we are going to use are studied in Chenevier's paper \cite{Ch3}. The first result that we need is the analogue of the results in \cite{finslope}, where the corresponding eigenvarieties were studied in \cite{BellaicheChenevier}. 
We recall the set up of Chenevier's paper.
\bigskip

\begin{notation}\label{Notationseigenvar}
\begin{enumerate}
\noindent  \item[(i)] We choose a totally real field $F$ such that $[F:\Q]$ is even and let $E$ be a CM quadratic extension of $F$. We write $c$ for the complex conjugation of $E$ over $F$ and assume that there is a place $v_0$ of $F$ dividing $p$ such that $v_0=w_0w_0^c$ splits in $E$ and such that $F_{v_0}=E_{w_0}\cong K$. We fix such an isomorphism and view the uniformizer $\varpi$ of $K$ {as} an uniformizer of $F_{v_0}$. 

\item[(ii)] We fix an algebraic closure $\bar\Q$ of $\Q$ and embeddings $\iota_\infty:\bar\Q\hookrightarrow \C$ and $\iota_p:\bar\Q\hookrightarrow \bar\Q_p$. Let $I_\infty=\Hom(F,\C)=\Hom(F,\R)$ denote the set of infinite places of $F$. 
Given a place $v$ of $F$ dividing $p$ the set $I(v)=\Hom(F_v,\bar\Q_p)$ is identified with a subset $I_\infty(v)\subset I_\infty$ via our choice of embeddings $\iota_\infty$ and $\iota_p$. 

\item[(iii)] Let $d\geq1$ be an integer and let us write $G$ for the unique unitary group in $d$ variables defined over $F$ which splits over $E$, is quasi-split at all finite places and compact at all infinite places. The existence of such a group can be deduced from the considerations of section $2$ of \cite{Clozelautoduales}. 

\item[(iv)] As $v_0$ splits in $E$, there exists an isomorphism $G(F_{v_0})\cong \GL_d(K)$ that we fix for the following. We write $S_p$ for the set of places $v$ of $F$ dividing $p$ and $S_p'=S_p\backslash\{v_0\}$\footnote{Let us remark that here $S_p$ is not exactly the same as in \cite{Ch3}}.


\item[(v)] 
Let $T$ denote the diagonal torus in $\GL_d(K)$ and denote by $T^0$ its maximal compact subgroup. 
Further we fix the Borel $B\subset\GL_d(K)$ of upper triangular matrices in order to have a notion of dominant weights.
Let $L\subset \bar\Q_p$ be a subfield containing $\sigma(F_{v_0})$ for all $\sigma\in I(v_0)$.
We define the weight space for the automorphic representations to be 
\[\Wcal^{\rm aut}=\Hom_{\rm cont}(T^0,\Gbb_m(-)),\]
as a rigid space over $L$. Especially we have a canonical identification $\Wcal^{\rm aut}\cong \Wcal_L^d$.


\item[(vi)] Fix a finite set $S$ of finite places of $F$ containing $S_p$ and all places such that $G(F_v)$ ramifies and fix a compact open subgroup $H=\prod_{v}H_v\subset G(\Abb_{F,f})$ such that $H_v$ is maximal hyperspecial for all $v\notin S$ and such that $H_{v_0}$ is $\GL_d(\Ocal_K)$. Write $S'=S\backslash\{v_0\}$. We define $H'=\prod_v H'_v$ such that $H'_v=H_v$ is $v\neq v_0$ and $H'_{v_0}$ is the Iwahori-subgroup $I$ of $\GL_d(\Ocal_K)$ of matrices whose reduction modulo $\varpi_K$ are upper triangular. Further let $\Hcal^{\rm un}=\Ocal_L[G(\Abb_{F,f}^S)/\hspace{-.1cm}/H^S]$ denote the spherical Hecke-algebra outside of $S$. Furthermore, we ask that $H$ is small enough, ie for $g\in G(\Abb_{F,f})$,
\begin{equation*}
G(F)\cap gHg^{-1}=1.
\end{equation*}

\item[(vii)] For each place $v\in S'$ we fix an idempotent element $e_v$ in the Hecke-algebra $\Ocal_L[G(F_v)/\hspace{-.1cm}/H_v]$ and write $e=(\otimes_{v\in S'}e_v)\otimes 1_{\Hcal^{\rm un}}$ for the resulting idempotent element of the Hecke algebra $\Ocal_L[G(\Abb_{F,f}^{v_0})/\hspace{-.1cm}/H^{v_0}]$.

\item[(viii)] For $1\leq i\leq d$, let $t_{i}={\rm diag}(1,\dots,1, \varpi,1\dots,1)\in T$, where the uniformizer is at the $i$-th diagonal entry. 
Let $T^-\subset T$ denote the set of ${\rm diag}(x_1,\dots, x_d)\in T$ such that ${\rm val}(x_1)\geq \dots \geq {\rm val}(x_d)$. 
We regard $\Z[T/T^0]$ as a subring of the Iwahori-Hecke algebra of $G(F_{v_0})$ with coefficients in $\Z[1/p]$ by means of $t\mapsto \mathbbm{1}_{H_{v_0}tH_{v_0}}$. This subalgebra is generated by the Hecke-operators $\mathbbm{1}_{H_{v_0}tH_{v_0}}$ for $t\in T^-$ and their inverses. 
Finally let $\Hcal=\Hcal^{\rm un}\otimes_\Z \Z[T/T^0]$, which is a subalgebra of $L[G(\Abb_{F,f})/\hspace{-.1cm}/H']$.
\end{enumerate}
\end{notation}

Let $W_\infty$ be an irreducible algebraic representation of $\prod_{v\in S'_p, w\in I_\infty(v)}G(F_w)$ and let $\Acal=\Acal(W_\infty,S,e)$ denote the set of {isomorphism classes of} all irreducible automorphic representations $\Pi$ of $G(\Abb_F)$ such that $\bigotimes_{v\in S'_p, w\in I_\infty(v)}\Pi_w$ is isomorphic to $W_\infty$ and $e(\Pi_f)^{H'_{v_0}}\neq 0$.
Further define the set of classical points to be 
\begin{equation}\label{classicalpoints}
\Zcal=\left\{(\Pi,\chi)\left|
\begin{array}{*{20}c}
 \Pi\in \Acal, \chi: T/T^0\rightarrow \bar\Q_p^\times \ {\rm continuous}\\ \text{such that}\ {\Pi_{v_0}|\det|_{v_0}^{\frac{1-d}{2}}}\ \text{is a sub-object of}\ {\rm Ind}_B^{\GL_d(K)}\chi
\end{array}\right.\right\}
\end{equation}
where the parabolic induction is normalized.

Associated to these data there is an \emph{eigenvariety}, that is a reduced rigid analytic space $Y(W_\infty,S,e)$ over $L$ together with a morphism 
\[\kappa: Y(W_\infty,S,e)\longrightarrow \Wcal^{\rm aut}\] and 
\[\psi=\psi^{\rm un}\otimes\psi_{v_0}:\Hcal\longrightarrow \Gamma(Y(W_\infty,S,e),\Ocal_{Y(W_\infty,S,e)})\]
a morphism of algebras such that $Y(W_\infty,S,e)$ contains a set $Z$ as a Zariski-dense accumulation\footnote{Recall that a subset $A\subset Y$ of a rigid space accumulates at a point $x\in Y$ if $A\cap U$ is Zariski-dense in $U$ for every connected affinoid neighborhood $U$ of $x$ in $Y$.  
} subset. These data are due to the property that there is a bijection between $Z$ and $\Zcal$ sending a point $z\in Z$ on the pair $(\Pi_z, \chi_z)\in \Zcal$ according to the following rule.

The evaluation $\psi^{\rm un}(z):\Hcal^{\rm un}\rightarrow k(z)$ is the character of the spherical Hecke-algebra associated to the representation $\Pi_z^S$. 
For $w\in I_\infty(v_0)$, let $\kappa_{\Pi_{z,w}}$ denote the algebraic character of $T_{v_0}$ obtained from $\Pi_{z,w}$ following the rule of \cite[\S$1.4$]{Ch3}. Then, $\kappa(z)=\prod_{w\in I_\infty(v_0)}\kappa_{\Pi_{z,w}}$. Let $\kappa_{\varpi}(z)$ be the unique character $T/T^0\rightarrow\bar\Q_p^{\times}$ such that $\kappa_{\varpi}(z)(t)=\kappa(z)(t)$ when $t$ is a diagonal matrix whose entries are powers of $\varpi_K$. 
Finally the component $\psi_{v_0}$ of the morphism $\psi$ is given by 
\[\psi_{v_0}(z)|_{T_{v_0}^-}:\mathbbm{1}_{H_{v_0}tH_{v_0}}\longmapsto \chi_{z}(t)\cdot\delta_{B_v}^{-1/2}(t)|\det(t)|^{\tfrac{d-1}{2}}\kappa_{\varpi}(z)(t),\]
where $\delta_{B_v}$ is the modulus character.

In what follows, we fix the data $(W_\infty,S,e)$ and write simply $Y$ for $Y(W_\infty,S,e)$.

In \cite[\S$2$]{Ch3}, Chenevier constructs these eigenvarieties using a space of overconvergent $p$-adic automorphic forms. More precisely, if $V$ is an open affinoid of $\Wcal^{\rm aut}$, one defines a certain $r_V\geq1$, and constructs for each $r\geq r_V$ an $\Ocal(V)$-Banach space denoted $e\Scal(V,r)$ with a continuous action of $\Hcal$ such that the operator $U_{v_0}$ corresponding to ${\rm diag}(\varpi_K^{d-1},\varpi_K^{d-2},\dots,1)\in\Z[T/T^0]\subset\Hcal$ acts as a compact operator. We say that a character of $\Hcal$ is $U_{v_0}$-finite if the image of $U_{v_0}$ is non zero. Then we have the following interpretation of points of $Y$, which is a consequence of Buzzard's construction of eigenvarieties \cite[\S$5$]{Buzzard}.

\begin{prop}\label{points}
Let $t\in\Wcal(\bar\Q_p)$. Then there is a natural bijection between $\bar\Q_p$-points of $Y$ mapping to $t$ and $\bar\Q_p$-valued $U_{v_0}$-finite system of eigenvalues of $\Hcal$ on $\varinjlim_{V,r}e\Scal(V,r)\otimes_{\Ocal(V)}k(t)$.
\end{prop}


\subsection{The map to the finite slope space}\label{eigenvar}
In the above section we have recalled the construction of an eigenvariety $Y\rightarrow \Wcal^{\rm aut}$. 
As above we write $Z\subset Y$ for the set of classical points $(\ref{classicalpoints})$. 
Let $(\Pi,\chi)\in Z$ and let $\pi=\bigotimes'_{v}{\rm BC}(\Pi_v)$ be the representation of $\GL_d(\Abb_E)$ defined by local base change for $\GL_d$. Then by \cite[Theorem 3.2, 3.3]{Ch3} there are Galois-representations $\rho_{\Pi}:\Gcal_E\rightarrow \GL_n(\bar\Q_p)$ attached to the automorphic representations $\Pi\in Z$ such that the Weil-Deligne representation attached to $\rho_\Pi|_{\Gcal_v}$ equals the Langlands parameter of $\pi_v |\cdot|^{(1-d)/2}$, where $\Gcal_v\subset \Gcal_E$ is the decomposition group at $v$ for $v$ not dividing $p$.

Let $\mathbb{B}\subset \mathbb{G}=\Res_{K/\Q_p}\GL_d$ denote the Weil restriction of the Borel subgroup of upper triangular matrices and let $\mathbb{T}\subset \mathbb{B}$ denote the Weil restriction of the diagonal torus. 
Using the canonical isomorphism $\mathbb{G}_{\bar\Q_p}\cong \prod_\sigma \GL_{d,\bar\Q_p}$ an algebraic weight ${\bf n}$ of $(\Gbb_{\bar\Q_p},\mathbb{T}_{\bar\Q_p})$ that is dominant with respect to $\mathbb{B}_{\bar\Q_p}$ can be identified with a tuple $(n_{\sigma,1},\dots, n_{\sigma,d})_{\sigma\in I(v_0)}\in \prod_{\sigma\in I(v_0)}\Z^d$ such that $n_{\sigma,1}\geq\dots\geq n_{\sigma,d}$ for all $\sigma$. 
Note that this algebraic weight is already canonically defined over the reflex field $E_{\bf n}$ of the weight ${\bf n}$, i.e.~over the subfield of $\bar\Q_p$ defined by
\[{\rm Gal}(\bar\Q_p/E_{\bf n})=\{\psi\in{\rm Gal}(\bar\Q_p/\Q_p)\mid n_{\sigma,i}=n_{\psi\circ\sigma,i}\ \text{for all embeddings}\ \sigma\}\] and hence especially over our fixed field $L$. Especially ${\bf n}$ defines an $L$-valued point of $\Wcal^d$.
 
Let $z=(\Pi,\chi)\in Z$ and for $\sigma\in I(v_0)$ let $n_{\sigma,1}\geq \dots \geq n_{\sigma,d}$ denote the highest weight of $\Pi_{v(\sigma)}$, where $v(\sigma)=\iota_\infty\iota_p^{-1}\sigma\in I_\infty(v_0)$. We say that $z$ is regular (with respect to $v_0$) if $n_{\sigma,1}> \dots > n_{\sigma,d}$ for all $\sigma\in I(v_0)$ and if 
\[\tfrac{\lambda_i}{\lambda_j}\notin \{1,p^{\pm [K_0:\Q_p]}\},\]
where we set
\[\lambda_i={\chi(t_{v_0,i})}{=\psi_{v_0}(t_{v_0,i})\cdot (|N_{K/\Q_p}(\varpi)|N_{K/\Q_p}(\varpi))^{i-1}}.\]
We further say that $z=(\Pi,\chi)$ is uncritical if in addition condition $(\ref{noncrit})$ holds {with $\lambda_i$ as above and $k_{\sigma,i}=n_{\sigma,i}-(i-1)$.}
We write $Z^{\rm reg}\subset Z$ for the set of regular points and $Z^{\rm un}\subset Z$ for the set of uncritical regular points. 
\begin{lem}\label{noncritdense}
The subsets $Z^{\rm reg}$ and $Z^{\rm un}$ are Zariski-dense in the eigenvariety $Y$ and accumulate at all classical points $z\in Z$.
\end{lem}
\begin{proof}
The proof is the same as the usual proof of density of classical points.
Let us denote by $Y_0\subset \Wcal^d\times\Gbb_m$ the Fredholm hypersurface cut out by the Fredholm determinant of $U_{v_0}={\rm diag}(\varpi^{d-1},\dots,\varpi,1)$.
Let $z\in Z\subset Y$ be a classical point and let $U\subset Y$ be a connected affinoid neighborhood. 
After shrinking $U$ we may assume that there is an affinoid open subset $V\subset \Wcal^d$ such that $U\rightarrow V$ is finite and torsion free. 
As $U$ is quasi-compact, there exist $C_1,\dots, C_d$ such that
\[C_i\geq {\rm val}_x({\lambda_1\lambda_2\cdots\lambda_d}(x))+1\]
for all $x\in U$, where ${\rm val}_x$ is the valuation on $k(x)$ normalized by ${\rm val}_x(p)=1$.  
Let us write $A\subset V$ for the set of dominant algebraic weights $n_{\sigma,1}\geq \dots\geq n_{\sigma,d}$ such that $C_i< n_{\sigma,i}-n_{\sigma,i+1}+1$ for all $i$ and $\sigma:K\hookrightarrow \bar\Q_p$ and 
\begin{equation}\label{Zunweightsacc}
\begin{aligned}
\tfrac{[K:\Q_p]}{[K_0:\Q_p]}\cdot C_1-\sum\nolimits_{\sigma'} n_{\sigma',1}<&{-n_{\sigma,2}+1-\sum\nolimits_{\sigma'\neq\sigma} n_{\sigma',1}} \ \text{for all embeddings}\ \sigma\\
\tfrac{[K:\Q_p]}{[K_0:\Q_p]}C_i -\sum_{j=1}^i\sum\nolimits_{\sigma'} n_{\sigma',j} <&{-n_{\sigma,i+1}+i+\sum\nolimits_{\sigma'\neq\sigma} (-n_{\sigma',i}+i-1)}\\&{+\sum\nolimits_{\sigma'}\sum_{j=1}^{i-1}(-n_{\sigma',j}+j-1)}\ \text{for all}\ \sigma,i.
\end{aligned}
\end{equation}
Then one easily sees that $A$ accumulates at the point $\kappa(z)$. It follows from \cite[Theorem 1.6 (vi)]{Ch3} that the points $z'\in U$ such that $\kappa(z')\in A$ are classical, whereas $(\ref{Zunweightsacc})$ assures that these points lie in $Z^{\rm un}$, as
\[\tfrac{[K:\Q_p]}{[K_0:\Q_p]}{\rm val}(\lambda_1)<\tfrac{[K:\Q_p]}{[K_0:\Q_p]}\big(C_1-\sum\nolimits_\sigma n_{\sigma,i}{\rm val}_x(\sigma(\varpi))\big)=\tfrac{[K:\Q_p]}{[K_0:\Q_p]}C_1-\sum\nolimits_\sigma n_{\sigma,1}\]
and
\[-k_{\sigma,2}-\sum\nolimits_{\sigma\neq\sigma'}k_{\sigma',1}={-n_{\sigma,2}+1-\sum\nolimits_{\sigma\neq\sigma'}n_{\sigma',1}}\]
for all $\sigma$ and similarly for the second required inequation.
The claim now follows from this as the map $U\rightarrow V$ is finite and torsion free. 
\end{proof}
Let us fix an identification of the decomposition group $\Gcal_{w_0}$ of $\Gcal_E$ at $w_0$ with the local Galois group $\Gcal_K$.
\begin{prop}\label{locglob}
Let $\Pi=\Pi_z$ for some $z\in Z^{\rm reg}$. For an infinite place $v\in I$ let  $n_{v,1}\geq\dots\geq n_{v,d}$ denote the highest weight of $\Pi_v$.
Then the representation $\rho_\Pi|_{\Gcal_K}$ is crystalline with Hodge-Tate weights\footnote{Again note that we use a different sign convention as \cite{Ch3}.} 
\begin{equation}\label{fromntok}
{k_{\sigma,i}=n_{v(\sigma),i-1}-(i-1)},
\end{equation}
where $v(\sigma)=\iota_\infty\iota_p^{-1}\sigma$. Moreover the Frobenius $\Phi_{{\rm cris},\Pi}$ that is the $[K_0:\Q_p]$-th power of the crystalline Frobenius on 
\[\WD(\rho_\Pi|_{\Gcal_K})=D_{\rm cris}(\rho_\Pi|_{\Gcal_K})\otimes_{K_0\otimes_{\Q_p}\bar\Q_p}\bar\Q_p\]
is semi-simple, its eigenvalues are distinct and given by $\lambda_i={\chi_z(t_i)}$. 
\end{prop}
\begin{proof}
It follows from \cite[Theorem 3.2]{Ch3} that the representation is semi-stable with Hodge-Tate weights and Frobenius eigenvalues as described above. The condition 
\[\tfrac{\lambda_i}{\lambda_j}\neq p^{\pm [K_0:\Q_p]}\]
assures that the monodromy operator has to vanish and hence the representation is crystalline. Further the condition $\lambda_i/\lambda_j\neq 1$ assures that the Frobenius has distinct eigenvalues and is a priori semi-simple.
\end{proof}
By \cite[Corollary 3.9]{Ch3} there is a pseudo-character $T_Y:\Gcal_{E,S}\rightarrow \Gamma(Y,\Ocal_Y)$ such that for all $\Pi\in Z^{\rm reg}$ one has $T\otimes k(\Pi)={\rm tr}\, \rho_\Pi$.
Let us write $\Gcal_{E,S}$ for the Galois group of the maximal extension $E_{S}$ inside $\bar\Q$ that is unramified outside $S$ and fix a continuous residual representation $\bar\rho:\Gcal_{E,S}\rightarrow \GL_d(\Fbb)$ with values in a finite extension $\Fbb$ of $\Fbb_p$ such that the restriction $\bar\rho_{w_0}=\bar\rho|_{\Gcal_{w_0}}$ is absolutely irreducible. 
We write $R_{\bar\rho,S}$ reps.~$R_{\bar\rho_{w_0}}$ for the universal deformation rings of $\bar\rho$ resp.~$\bar\rho_v$ and let $\Xfrak_{\bar\rho,S}$ resp.~$\Xfrak_{\bar\rho_v}$ denote their rigid analytic generic fibers.
As we assume $\bar\rho_{w_0}$ (and hence also $\bar\rho$) to be absolutely irreducible \cite[Theorem A and Theorem B]{Ch2} implies that the universal deformation rings $R_{\bar\rho_{w_0}}$ and $R_{\bar\rho,S}$ agree with the universal deformation rings of the corresponding pseudo-characters ${\rm tr}\,\bar\rho_{w_0}$ resp.~${\rm tr}\,\bar\rho$.

Let $Y_{\bar\rho}\subset Y$ denote the open and closed subset where the pseudo-character $T_Y$ has residual type $\bar\rho$. Then the restriction to the decomposition group $\Gcal_K\cong\Gcal_{w_0}\subset \Gcal_{E,S}$ at $w_0$ induces a map $f_{\bar\rho}:Y_{\bar\rho}\rightarrow \Xfrak_{\bar\rho_{w_0}}$. 
Let $N_{K/\Q_p}:K^\times\rightarrow\Q_p^\times$ denote the norm map of $K$. We define $g_i:Y\rightarrow \Gbb_m$ by 
\[z\longmapsto \psi_{v_0}(t_{v_0,i})\cdot (|N_{K/\Q_p}(\varpi)|N_{K/\Q_p}(\varpi))^{i-1}.\]
Further we define a morphism
\[\omega_Y=(\omega_{Y,i})_i:Y\longrightarrow \Wcal^d\]
by setting $\omega_{Y,i}={\kappa_{i}\cdot\delta_{\Wcal}((1-i,\dots,1-i))}$.

\begin{theo}\label{maptofinslopespace}
The map
\[f=(f_{\bar\rho}, \omega_Y, (g_i)_i):Y_{\bar\rho}\longrightarrow \Xfrak_{\bar\rho_{w_0}}\times \Wcal\times \Gbb_m^d=\Xfrak_{\bar\rho_{w_0}}\times\Tcal^d\]
factors over the finite slope space $X(\bar\rho_{w_0})\subset \Xfrak_{\bar\rho_{w_0}}\times\Tcal^d$ and fits into the commutative diagram
\[\begin{xy}
\xymatrix{
Y_{\bar\rho}\ar[r]^f\ar[rd]_{\omega_Y} &  X(\bar\rho_{w_0})\ar[d]^{\omega_d}\\ & \Wcal^d
}\end{xy}\] 
\end{theo}
\begin{proof}
The subset $X(\bar\rho_{w_0})\subset \Xfrak_{\bar\rho_{w_0}}\times\Tcal^d$ is Zariski-closed an hence it suffices to check that $f(z)\in X(\bar\rho_{w_0})$ for all $z=(\Pi_z,\chi_z)\in Z^{\rm un}\cap Y(\bar\rho)$, as this subset is Zariski-dense by Lemma $\ref{noncritdense}$ and as $Y_{\bar\rho}$ is reduced. By Lemma $\ref{nsvssquare}$ this amounts to say that for $z\in Z^{\rm un}$ the representation $\rho_{\Pi_z}|_{D_{w_0}}$ is trianguline with graded pieces $\Rcal(\delta_i)$, where $\delta_i:K^\times\rightarrow \bar\Q_p^\times$ is the character
\begin{align*}
\delta_i|_{\Ocal_K^\times}:z\longmapsto { \prod\nolimits_{\sigma}\sigma(z)^{n_{v(\sigma),i}+1-i}}\\ 
\delta_i(\varpi) = {\psi_{v_0}(z)(t_{v_0,i})(|N_{K/\Q_p}(\varpi)|N_{K/\Q_p}(\varpi))^{1-i}}.
\end{align*} 
where as above $v(\sigma)=\iota_\infty\iota_p^{-1}\sigma$ and where we write $(n_{v(\sigma), i})$ for the highest weight of $\Pi_{z,v(\sigma)}$.
By our choice of $Z^{\rm un}$ this follows from Lemma $\ref{nctrianggradedpieces}$ and Proposition $\ref{locglob}$. 
\end{proof}

\subsection{A density result for the space of $p$-adic automorphic forms}\label{densityautsection}

We now introduce the Banach space of $p$-adic automorphic forms of tame level $H^{v_0}$ and prove that a an element of $R_{\bar\rho,S}$ vanishing on this space, vanishes on the eigenvariety $Y(W_\infty,S,e)_{\bar\rho}$ too.

Recall that we have fixed a finite extension $L$ of $\Q_p$ with ring of integers $\Ocal$ and uniformizer $\varpi_L$. If $H=\prod_v H_v\subset G(\Abb_{F,f})$ is a compact open subgroup such that $H_v\subset \GL_d(F_v)$ for $v|p$, we can define, for $W_0$ a finite $\Ocal$-module with a continuous action of $\GL_d(\Ocal_F\otimes\Z_p)$, {the space of automorphic forms of level $H$ and weight $W_0$ by}
\begin{equation*}
S_{W_0}(H,\mathcal{O})=\{f:\, G(F)\backslash G(\mathbb{A}_F^{\infty})\rightarrow W_0\mid f(gh)=h^{-1}f(g)\ \text{for all}\ h\in H\}.
\end{equation*}
If $H^{v_0}=\prod_{v\neq v_0} H^v$, we can define
\begin{equation*}
S_{W_0}(H^{v_0},\Ocal)=\varinjlim_{H_{v_0}\subset G(\Ocal_{F_{v_0}})}S_{W_0}(H^{v_0}H_{v_0},\mathcal{O}),
\end{equation*}
where the limit is taken over all compact open subgroups of $G(\Ocal_{F_{v_0}})$.

Let $\hat S_{W_0}(H^{v_0},\mathcal{O})$ be the $\varpi_L$-adic completion of $S_{W_0}(H^{v_0},\mathcal{O})$. When $W_0$ is the trivial representation, we omit it from the notation.

If ${{\bf n}}$ is a dominant algebraic weight, we write $\mathbb{W}_{{{\bf n}}}$ for the irreducible representation of $\Gbb_{\bar\Q_p}=(\Res_{K/\Q_p}\GL_d)_{\bar\Q_p}$ of highest weight ${{\bf n}}$ relatively to our choice of Borel subgroup. 
Note that this representation is already canonically defined over the reflex field $E_{{{\bf n}}}$ of the weight ${{\bf n}}$ and hence especially over our fixed field $L$, because we assumed that $L$ contains all the Galois conjugates $\sigma(K)$ of $K$ inside $\bar\Q_p$. 
Finally we write $W_{{{\bf n}}}$ for the representation of $\GL_d(K)$ or $\GL_d(\Ocal_K)$ given by composing the embedding 
\[\begin{aligned}\GL_d(K)&\longrightarrow \prod\nolimits_\sigma \GL_d(L)=(\Res_{K/\Q_p}\GL_d)(L)\\ x&\longmapsto (\sigma(x))_\sigma.\end{aligned}\]
with the evaluation of $\mathbb{W}_{{{\bf n}}}$ on $L$-valued points (and similar for its restriction to $\GL_d(\Ocal_K)$).

Recall that $W_\infty$ is the representation of $\prod_{v\in S'_p}\prod_{w\in I_\infty(v)}G(F_v)$ fixed in section \ref{eigenvarieties}. Using $\iota_p$ and $\iota_\infty$ and choosing $L$ big enough, we can view $W_\infty$ as a representation of $\prod_{v\in S_p'}G(F_v)$ and put an $L$-structure on it. Let us write $\hat S_{\bf{n}}(H^{v_0},L)=\hat S_{W^0}(H^{v_0},\Ocal)\otimes_{\Ocal} L$, where $W^0$ is a stable $\Ocal_L$-lattice of the representation $W_{\bf n}\otimes_{\Ocal_L}W_\infty$ of $G(\Ocal_F\otimes\Z_p)=\GL_d(\Ocal_K)\times\prod_{v\in S'_p}G(\Ocal_{F_v})$.

If $H$ is a compact open subgroup of $G(\mathbb{A}_f)$ we write $\Hcal(H)$ for the image of $\Hcal^{\rm un}$ in $\mathrm{End}(\hat S_{\bf k}(H,L))$. 

Now we fix $H$ as in section \ref{eigenvarieties}, and assume moreover now that all places $v|p$ are split in $E$ and $H_v$ is maximal at theses places\footnote{This restriction is only here to be able to apply the idempotent $e$ at the spaces $S_W(H,L)$.}. Recall that we fixed a Galois representation $\bar\rho$ which is \emph{automorphic} of level $H$, i.e.~there exists $z\in Z$ such that $\bar\rho$ is isomorphic to the reduction mod $\varpi_L$ of $\rho_{\Pi_z}$. Let $\mfrak$ be the maximal ideal of $\Hcal^{\rm un}$ such that for $v\notin S$, the conjugacy class of $\bar\rho({\rm Frob}_v)$ coincides via the Satake correspondence with the morphism $\Hcal(G(F_v),H_v)=\Ocal_L[G(F_v)/\hspace{-.1cm}/H_v]\rightarrow\Hcal(H)/\mfrak\simeq k_L$. 

Given a compact open subgroup $H_{v_0}\subset G(\Ocal_{F_{v_0}})$ we write $\Hcal_{\mfrak}(H_{v_0})$ for the image of $\Hcal^{\rm un}_{\mfrak}$ in $\mathrm{End}(S(H^{v_0}H_{v_0},\Ocal)_{\mfrak})$.
Then there is a unique continuous map \[\theta:\,R_{\bar\rho}\longrightarrow\Hcal_\mfrak(H_{v_0})\] with the following property: Given a morphism $\psi:\,\Hcal_\mfrak(H_{v_0})\rightarrow\bar\Q_p$ of $\Ocal_L$-algebras, the deformations $\rho$ of $\bar\rho$ corresponding to $\psi\circ\theta$ are such that for $v\notin S$, the conjugacy class of $\rho({\rm Frob}_v)$ coincides via the Satake correspondence with the morphism \[\Hcal(G(F_v),H_v)\longrightarrow\Hcal_\mfrak(H_{v_0})\xrightarrow{\psi}\bar\Q_p.\] By unicity, these maps glue into a map \[\theta:R_{\bar\rho}\longrightarrow \varprojlim_{H_{v_0}}\Hcal_{\mfrak}(H_{v_0})\] giving a continuous action of $R_{\bar\rho}$ on $\hat S(H^{v_0},\Ocal)_{\mfrak}$.


Now we can fix an idempotent $e$ as in section \ref{eigenvarieties} such that $e{S(H^{v_0},L)}_\mfrak\neq0$. We will prove that if an element $t\in R_{\bar\rho}$ vanishes on $e\hat S(H^{v_0},L)_{\mfrak}$, then it vanishes on $Y_{\bar\rho}=Y(W_\infty,S,e)_{\bar\rho}$ too.

\begin{lem}\label{Wout}
{If ${\bf n}$ is a dominant algebraic weight, }there exists a $\GL_d(\Ocal_K)\times\Hcal^{\rm un}$-equivariant homeomorphism
\begin{equation*}
\hat S_{{{\bf n}}}(H^{v_0},L)_{\mfrak}\simeq W_{{{\bf n}}}\otimes_L\hat S(H^{v_0},L)_{\mfrak},
\end{equation*}
supposing that $\Hcal^{\rm un}$ acts trivially on $W_{{{\bf n}}}$.
\end{lem}

\begin{proof}
It is sufficient to prove it before the localization in $\mfrak$, by $\Hcal^{\rm un}$-equivariance. Then we can use the following list of $\GL_d(\Ocal_K)\times\Hcal^{\rm un}$-equivariant isomorphisms
\begin{align*}
W^0_{{{\bf n}}}\otimes_{\Ocal_L} \hat S(H^{v_0},\Ocal)&=\varprojlim_n(W^0_{{{\bf n}}}/\varpi_L^n)\otimes_{\Ocal_L}S(H^{v_0},\mathcal{O}/\varpi_L^n))\\
&=\varprojlim_n(W^0_{{{\bf n}}}/\varpi_L^n)\otimes_{\Ocal_L}(\varinjlim_{H_{v_0}}S(H_{v_0}H^{v_0},\mathcal{O}/\varpi_L^n))\\
&=\varprojlim_n \varinjlim_{H_{v_0}}(S_{W^0_{{{\bf n}}}/\varpi_L^n}(H_{v_0}H^{v_0},\mathcal{O}/\varpi_L^n)\\
&=\hat S_{{{\bf n}}}(H^{v_0},\mathcal{O}_L).
\end{align*}
\end{proof}


%

\begin{prop}
The $\GL_d(\mathcal{O}_{K})$-representation $\hat S(H^{v_0},\mathcal{O})_{\mfrak}$ is isomorphic to a direct factor of $C(G(\mathcal{O}_{K}),L)^r$ for some $r\geq0$, where $C(G(\Ocal_K),L)$ denotes the space of continuous $L$-valued functions on $G(\Ocal_K)$.
\end{prop}

\begin{proof}
Using Lemma \ref{Wout}, it is sufficient to prove it when ${{{\bf n}}}={\bf 0}$. In this case we remark that the Banach space ${\hat S(H^{v_0},L)=}\hat S_{\bf 0}(H^{v_0},L)$ is the Banach space of continuous functions $G(F)\backslash G(\Abb_{F,f})/H^{v_0}\rightarrow W_\infty$. Let $g_1,\dots,g_{r'}\in G(\Abb_{F,f})$ be a set of representatives of $G(F)\backslash G(\Abb_{F,f})/H$. We have $G(F)\cap g_iHg_i^{-1}=\{1\}$ for each $i$, proving that $G(F)\backslash G(\Abb_{F,f})/H^{v_0}$ is isomorphic to $\GL_d(\Ocal_K)^{r'}$. This proves that ${\hat S(H^{v_0},L)}$ is $\GL_d(\Ocal_K)$-equivariantly isomorphic to $C(\GL_d(\Ocal_K),L)^r$ with $r=r'\dim W_\infty$. Using the fact that $\varprojlim_{H_{v_0}}\Hcal(H_{v_0},\Ocal_L)$ and its action on $\hat S(H^{v_0},L)$ commutes to $\GL_d(\Ocal_K)$, we can conclude that ${\hat S(H^{v_0},L)}_{\mfrak}$ is isomorphic to a direct factor of $C(\GL_d(\Ocal_K),L)^r$.
\end{proof}

By Lemma \ref{Wout}, there is an $\Hcal^{\rm un}$-equivariant isomorphism
\[S_{{{\bf n}}}(H,L)_{\mfrak}\simeq {\Hom}_{\GL_d(\Ocal_K)}(W_{{{\bf n}}}^*,\hat S(H^{v_0},L)_{\mfrak}).\] This implies that if $t\in R_{\bar\rho}$ vanishes on ${\hat S}(H^{v_0},L)_{\bar\rho}$ it will vanish at each point of $Z\subset Y_{\bar\rho}$. These points being Zariski-dense in $Y$, the function $t$ vanishes on $Y_{\bar\rho}$.

%

\subsection{A density result for the eigenvariety}

Now we fix ${\bf k}$ a strongly dominant weight. We say that a closed point $y\in Y_{\bar\rho}$ is crystabelline of Hodge-Tate weights ${\bf k}$ if its image in $\Xfrak(\bar\rho)$ is crystalline on an abelian extension of $K$ and its Hodge-Tate weights are given by ${\bf k}$. The purpose of this section is to prove that if $t\in R_{\bar\rho}$ vanishes on the subset of points of $Y_{\bar\rho}$ which are crystabelline of Hodge-Tate weights ${\bf k}$, then $t$ vanishes on $Y_{\bar\rho}$.

Recall that we have fixed a Borel subgroup $B\subset \GL_d(K)$ and let us write $N\subset B$ for its unipotent radical. Further we write $N_0=N\cap \GL_d(\Ocal_K)$.

%
%
%
%

\begin{defn}
Let $\Pi$ be an irreducible smooth representation of $G(F_{v_0})$. We say that $\Pi$ has finite slope if the operator $U_{v_0}$ has a non zero eigenvalue on the space $\Pi^{N_0}$. If $\Pi$ is an irreducible automorphic representation of $G(\mathbb{A}_F)$, we say that $\Pi$ has finite slope if $\Pi_{v_0}$ has finite slope as a smooth representation of $G(F_{v_0})$.
\end{defn}

The following result tells us that finite slope automorphic representations of $G(\mathbb{A}_F)$ give rise to closed points of $Y_{\bar\rho}$.

\begin{prop}
Let $\Pi$ be an irreducible automorphic representation of $G(\mathbb{A}_F)$ of finite slope {whose isomorphism class lies in $\Acal(W_\infty,S,e)$}. Then there exists a point $z\in {Y(W_\infty,S,e)}$ such that $\psi_z|_{\Hcal^{\rm un}}=\psi_{\Pi}|_{\Hcal^{\rm un}}$. Moreover, if $\bigotimes_{w\in I_{\infty}(v_0)}\Pi_w$ is isomorphic to $W_{{{\bf n}}}$, then the Galois-representation attached to $z$ becomes semi-stable of weight ${\bf k}{=(k_{\sigma,i})}$, when restricted to the Galois group of an abelian extension of $K$, {where $k_{\sigma,i}=n_{\sigma,i}-(i-1)$}. If moreover, for $\eta_i=\omega_{Y,i}(z)\delta_{\Wcal}((k_{\sigma,i}))^{{-1}}$, we have $\eta_i\neq\eta_j$ for $i\neq j$ then this Galois representation is even potentially crystalline of weight ${\bf k}$.
\end{prop}

\begin{proof}
By assumption, there exists $f\in{e\hat S(H^{v_0},L)}^{N_0}$ which is an eigenvector of $\Hcal\times L[T^0]$ such that the character of $\Hcal$ is $\psi_{\Pi}$ and the eigenvalue of $U_{v_0}$ is non zero. Let $\chi$ be the character of $T^0$ giving the action of $T^0$ on $f$. By \cite[Proposition 3.10.1]{Loeffleroverconvergent}, we have
\begin{equation*}
e{\hat S(H^{v_0},L)}^{N_0}[\chi]\simeq\varinjlim_r e\Scal(\chi,r)
\end{equation*}
By Proposition \ref{points}, there exists a point $z$ of $Y({W_\infty,S,e})$ such that $\psi_z|_{\Hcal^{\rm un}}=\psi_{\Pi}|_{\Hcal^{\rm un}}$. The claim about semi-stability after an abelian extension follows easily using the map to the finite slope space and the fact that the fibers over strongly dominant locally algebraic characters have this property. By Theorem \ref{maptofinslopespace}, the character $(\eta_i)_i$ gives the action of the inertia on the Weil-Deligne module of this Galois representation, which is non monodromic under the last assumption of the proposition.
\end{proof}

{Let $\bf{k}$ and $\bf{n}$ as in the proposition.} This proposition shows us that if we want to prove that an element $t\in\Hcal^{\rm un}$ vanishing on all crystabelline points of type ${\bf k}$ of $Y{(W_\infty,S,e)}_{\bar\rho}$ is zero, it is sufficient to prove the following: An element $t\in\Hcal^{\rm un}$ vanishing on ${e}S_{{\bf n}}(H^{v_0},L)_{\mfrak}[\Pi_f]$ for all irreducible automorphic representations $\Pi$ of finite slope such that $e\Pi_f\neq0$, satisfies $t=0$ on $e\hat S_{{\bf }}(H^{v_0},L)_\mfrak$.

To produce sufficiently automorphic finite slope representations we can use the following result. Now we write $I_n$ for the level $n$ Iwahori subgroup of $\GL_d(\Ocal_K)$ i.e.~the group of elements of $\GL_d(\Ocal_K)$ such that the entries below the diagonal are divisible by $\varpi^n$, and $B_0=B\cap I_n$ and $N_0=N\cap B_0$. Recall that the level of a smooth character $\chi:\,\mathcal{O}_K^{\times}\rightarrow\mathbb{C}^{\times}$ is the least integer $n$ such that $1+\varpi^{n+1}\mathcal{O}_K$ is contained in $\ker(\chi)$.

\begin{prop}
Let $\chi=\bigotimes_{i=1}^{d}\chi_i$ be a smooth character of $T^0$ such that for $1\leq i \leq n-1$, the level of $\chi_i$ is strictly bigger than the level of $\chi_{i+1}$, then there exists an open subgroup $I(\chi)$ such that $I(\chi)=(I(\chi)\cap\bar{N})T^0(I(\chi)\cap N)$, $I(\chi)\cap B=B^0$ and if we write $\chi$ for the composite $I(\chi)\rightarrow T^0\rightarrow\mathbb{C}^{\times}$, then the pair $(I(\chi),\chi)$ is a type for the inertial conjugacy class of $(T,\chi)$, more precisely, if $\pi$ is a smooth irreducible representation of $\mathrm{GL}_d(K)$, then \[\Hom_{I(\chi)}(\chi,\pi)\neq 0\Longleftrightarrow \pi\cong\mathrm{Ind}_B^{\mathrm{GL}_d(K)}(\eta)\] with $\eta$ a character of $T$ such that $\eta|_T=\chi$. Moreover, in this case, $\pi$ has finite slope.
\end{prop}
\begin{proof}
Let $n_i$ be the level of $\chi_i$ and define $I(\chi)$ as the subgroup of $I$ of matrices $(a_{i,j})_{1\leq i,j \leq d}$ such that $\varpi^{n_j}|a_{i,j}$ for $j< i$. It is immediate to check that $\chi$ can be extended to a character of $I(\chi)$. It is enough to prove that $(I(\chi),\chi)$ is a type for the $\GL_d(K)$-inertial equivalence class of $(T,\chi)$. In this aim, we use the characterization of part $2$ of the introduction of \cite{BKsemisimple}. The only non trivial condition is $(iii)$ of loc.~cit. We follow closely the arguments of \cite{BKsemisimple} where the situation is much more general. Let $z$ be the element of $T$ whose diagonal entries are $(\varpi^{n-1},\varpi^{n-2},\dots,\varpi,1)$ and $f_z$ the element of $\mathcal{H}(G,\chi)$ with support $I(\chi)zI(\chi)$ such that $f_z(z)=1$. We only have to prove that $f_z$ is an invertible element of $\mathcal{H}(G,\chi)$. Let $f_{z^{-1}}$ be the element of support $I(\chi)z^{-1}I(\chi)$ such that $f_{z^{-1}}(z^{-1})=1$. We want to prove that $g=f_{z^{-1}}*f_z$ has support in $I(\chi)$. The support of $g$ is contained in $I(\chi)z^{-1}I(\chi)zI(\chi)$. Now remark that 
\[I(\chi)zI(\chi)=\coprod\nolimits_{u\in I(\chi)/(I(\chi)\cap zI(\chi)z^{-1})}uzI(\chi)\]
 and that each class of $I(\chi)$ modulo $I(\chi)\cap zI(\chi)z^{-1}$ contains an element of $N\cap I(\chi)=N_0$, so that $I(\chi)z^{-1}I(\chi)zI(\chi)=I(\chi)z^{-1}N_0zI(\chi)\subset I(\chi)N I(\chi)$. By \cite[Prop. $11.1.2.$]{BH}, it is then sufficient to check that if an element $u\in N$ intertwines the character $\chi$, then $u\in N_0$. We can restrict us to the case $d=2$. Let $u=\left(\begin{smallmatrix}1&x\\0&1\end{smallmatrix}\right)$. Suppose that $x\notin \mathcal{O}_K$ and choose $n=n_1-v(x)$ with $n_i$ the level of $\chi_i$. If $u$ intertwines $\chi$, an easy computation shows us that we must have $\chi_1(a+x\varpi^nc)\chi_2(d-x\varpi^nc)=\chi_1(a)\chi_2(d)$ for each $(a,d,c)\in\mathcal{O}_K^{\times}\times\Ocal_K^\times\times\mathcal{O}_K$. As $n_1 \geq n_2+1$, we have $\chi_2(d-x\varpi^nc)=\chi_2(d)$, so that we have $\chi_1(1+x\varpi^nc)=1$ for all $c\in\mathcal{O}_K$, which contradicts the fact that $n_1=n+v(x)$ is the level of $\chi_1$.
\end{proof}

Let $\Ttil^0$ be the set of smooth characters $T^0\rightarrow\Cbb_p$ of the form $\chi_1\otimes\cdots\otimes\chi_d$ such that the level of $\chi_i$ is strictly bigger than the level of $\chi_{i+1}$ for $1\leq i \leq d-1$.

\begin{prop}\label{prop:gl1}
Let $B$ be the Banach space of continuous function $\mathcal{O}_K^{\times}\rightarrow\mathbb{C}_p$ and for $n\in\mathbb{N}$, let $B_{\geq n}$ denote the subspace generated by the characters $\mathcal{O}_K^{\times}\rightarrow\mathbb{C}_p^{\times}$ of finite level bigger than $n$. Then $B_{\geq n}$ is dense in $B$.
\end{prop}

\begin{proof}
As the space of smooth functions from $\mathcal{O}_K^{\times}\rightarrow\mathbb{C}_p$ is dense in $B$ and a basis of this space is given by the set of all characters of finite level, the closure of $B_{\geq n}$ in $B$ is a subspace of finite codimension. If it is strictly included in $B$, there exists a continuous map $\lambda:\,B\rightarrow\mathbb{C}_p$ which is $U$-equivariant for some open subgroup $U\subset\mathcal{O}_K^{\times}$ acting trivially on $\mathbb{C}_p$. Then $\lambda$ gives rise to a non trivial Haar measure on $U$ which can not exist. 
\end{proof}

\begin{cor}\label{cor:dense}
Let $C(T^0,\Cbb_p)$ denote the space of continuous $\C_p$-valued functions on $T^0$. Then the subspace of $C(T^0,\Cbb_p)$ generated by the elements of $\Ttil^0$ is dense.
\end{cor}


\begin{prop}\footnote{V.~Pa{\v{s}}k{\=u}nas informed us that he has more general versions of this result in his forthcoming work with M.~Emerton}
Let $C(N_0\backslash \GL_d(\Ocal_K),\Cbb_p)$ denote the space of $\C_p$-valued continuous functions on $\GL_d(\Ocal_K)$ which are left invariant under $N_0$.
Then the subspace \[\sum_{\chi\in \Tcal_d} {\rm Ind}_{I(\chi)}^{\GL_d(\Ocal_K)}(\chi)\subset C(N_0\backslash \GL_d(\Ocal_K),\Cbb_p)\] is dense.
\end{prop}

\begin{proof}
If $\chi\in\Ttil^0$, the character $\chi$ of $T^0$ uniquely extends to a character $\chi$ of $I(\chi)$ which is trivial on $I_n\cap N$ and $I_n\cap\bar{N}$. Let's name such a function a \emph{character function} for the moment. More generally for $g\in\GL_d(\Ocal_K)$, the function $\chi(\cdot g)$ of support $I(\chi)g^{-1}$ is named a \emph{right translated character function}. For $\chi\in\Ttil^0$, the space ${\rm Ind}_{I(\chi)}^{\GL_d(\Ocal_K)}(\chi)$ is exactly the subspace of $C(N_0\backslash \GL_d(\Ocal_K,\Cbb_p))$ generated by the right translated character functions. Let $f:\,\GL_d(\Ocal_K)\rightarrow\mathbb{C}_p$ be a continuous function, invariant on the left under $N_0$. We have to prove that we can approximate $f$ by right translated character functions. Let $g_1,\dots,g_r$ be a system of representatives of the quotient $I_1\backslash \GL_d(\Ocal_K)$. Let $f_i=f(\cdot\ g_i^{-1})|_{I_1}$, so that $f=\sum_{i=1}^r f_i(\cdot g_i)$. Now fix $1\leq i\leq r$ and $\epsilon >0$. As $I_1$ is compact, we can find $n\geq 1$, such that for $h\in I_n\cap\bar N$, we have $||f_i-f_i(\cdot\ h)||<\epsilon$. Let $h_1,\dots,h_s\in I_1$ be a system of representatives of $(I_n\cap \bar{N})\backslash (I_1\cap\bar{N})$, which is also a system of representatives of $I_n\backslash I_1$, and define $f_{i,j}=f_i(\cdot\ h_j^{-1})|_{I_n}$. Let $f'_{i,j}$ be the function on $I_n$ defined by $f'_{i,j}(nt\bar{N})=f_{i,j}(t)$ for $(n,t,\bar{N})\in (N\cap I_n)\times T^0\times (\bar{N}\cap I_n)$. As $(\bar{N}\cap I_n)$ is a normal subgroup of $(\bar{N}\cap I_1)$, we have $||f_{i,j}-f'_{i,j}||<\epsilon$. Using Corollary \ref{cor:dense}, for each $(i,j)\in [1,r]\times [1,s]$, we can find elements $\tilde{f}_{i,j}\in \Tcal_d$, such that $||f'_{i,j}|_{T^0}-\tilde{f}_{i,j}||<\epsilon$. Now we can write each $\tilde{f}_{i,j}$ as $\sum_k a_{i,j,k}\chi_{i,j,k}$ with $I(\chi_{i,j,k})\subset I_n$. We extend each $\chi_{i,j,k}$ to $I(\chi_{i,j,k})$ as previously described. As $I(\chi_{i,j,k})\subset I_n$, we can write $\tilde{f}_{i,j}$ as a finite sum of right translated character functions. If $\tilde{f}=\sum_{i,j} \tilde{f}_{i,j}(\cdot\ h_j g_i)$, we have $||f-\tilde{f}|| <\epsilon$, and $\tilde{f}$ is a finite sum of right translated character functions.
\end{proof}

Now we can conclude the proof.

\begin{prop}
Let $t\in R_{\bar\rho}$ such that $t$ is zero on each $eS_{{{\bf n}}}(H^{v_0}I(\chi),\Cbb_p)_{\mfrak}[\chi]$ such that $\chi\in\Ttil^0$, then $t=0$ on $e\hat{S}(H^{v_0},L)_{\mfrak}^{N_0}$.
\end{prop}

\begin{proof}
We know that $\sum_{\chi}{\rm Ind}_{I(\chi)}^{\GL_n(\Ocal_K)}(\chi)$ is dense in $C(N_0\backslash \GL_n(\Ocal_K),\Cbb_p)$ and that $\hat{S}_{{{\bf n}}}(H^{v_0},L)_{\mfrak}$ is isomorphic to a direct summand of $C(\GL_d(\Ocal_K),L)^r$ for some $r$. It follows that $S_1=\hat S_{{{\bf n}}}(H^{v_0},L)_\mfrak \hat \otimes_L \Cbb_p$ is isomorphic to a direct summand of $C(\GL_d(\Ocal_K),\Cbb_p)^r$. We can write $C(\GL_d(\Ocal_K),\Cbb_p)^r=S_1\oplus S_2$. 

As the functor $F=\bigoplus_{\chi}{\rm Hom}_{I(\chi)}(\chi,-)$ commutes with finite direct sums, we know that $F(S_1)\oplus F(S_2)$ is dense in $[C(\GL_d(\Ocal_K),\Cbb_p)^r]^{N_0}$. As the functor of $N_0$-invariants commutes with direct sums, we conclude that $F(S_1)\subset S_1^{N_0}$ must be dense.  By assumption, $t$ vanishes on $F(S_1)$, hence on $S_1^{N_0}$, which contains $\hat{S}_{{{\bf n}}}(H^{v_0},L)_{\mfrak}^{N_0}$. Finally we conclude by remarking that \[\hat{S}_{{{\bf n}}}(H^{v_0},L)_{\mfrak}^{N_0}=W_{{{\bf n}}}\otimes_L \hat{S}(H^{v_0},L)_{\mfrak}^{N_0}.\] 
 \end{proof}

\begin{cor}\label{Rdensityoneigenvar}
Let $f\in R_{\bar\rho}$ be a function vanishing on all points of $Y_{\bar\rho}$ which are crystabelline of Hodge-Tate weights ${\bf k}$. Then the image of $f$ in $\Gamma(Y_{\bar\rho},\Ocal_Y)$ is zero.
\end{cor}

\subsection{Conclusion}
Let us summarize what we have proven so far using eigenvarieties. The following definition will be useful. 
\begin{defn}
Let $X$ be a rigid space and $R$ be a ring together with a ring homomorphism $\psi:R\rightarrow \Gamma(X,\Ocal_X)$.\\
\noindent (i)
A subset $Z\subset X$ is called \emph{$R$-closed} if $Z=\{x\in X\mid \psi(f)(x)=0\ \text{for all}\ f\in I\}$ for some ideal $I\subset R$.\\
\noindent (ii)  A subset $U\subset X$ is called \emph{$R$-open} if its complement is $R$-closed.
\end{defn}
Further we have an obvious notion of the $R$-closure of some subset $Z\subset X$ and a notion of $R$-density. 

Let ${\bf k}=(k_{\sigma,i})_\sigma\in \prod_\sigma\Z^d$ be a strongly dominant algebraic weight and $X_{\bf k}(\bar \rho_{w_0})$ denote the $R_{\bar \rho_{v_0}}$-closure of the set of {crystabelline} points of $X(\bar\rho_{w_0})$ which have labelled Hodge-Tate weights ${\bf k}$. We have finally proved the following result.

\begin{theo}\label{ImageEV}
The image of $Y(W_\infty,S,e)_{\bar\rho}$ is contained in $X_{\bf k}(\bar \rho_{w_0})$.
\end{theo}

\section{The main theorem}

%

Let us fix a continuous {absolutely irreducible} representation $\bar r:\,\Gcal_K\rightarrow\GL_d(\Fbb)$. We need to embed our local situation into a global one. For this we use the results of the appendix of \cite{EmertonGee} and the following proposition.

\begin{prop}\label{potdiag}
Let $\bf{k}$ be a set of labelled Hodge-Tate weights. Then $\bar r$ has a potentially diagonal lift of Hodge-Tate weights $\bf{k}$.
\end{prop}

\begin{proof}
Let $K'$ be the unique unramified extension of degree $d$ of $K$. Then there exists a character $\eta$ of $\Gcal_{K'}$ such that $\bar r\simeq\mathrm{Ind}_{\Gcal_{K'}}^{\Gcal_K}\eta$. If ${\bf k}=(k_{\sigma,i})$ is a stronlgy dominant weight, let $\theta$ be a character $\Gcal_{K'}\rightarrow\bar\Q_p^\times$ such that the restriction of $\theta$ to the inertia group of $K'$ corresponds, via local class field theory, to the character of $\mathcal{O}_{K'}^\times$ given by $x\mapsto\prod_{\sigma,i}\psi_{\sigma,i}(x)^{k_{\sigma,i}}$, $(\psi_{\sigma,i})_{1\leq i\leq d}$ being the set of embeddings of $K'$ in $\bar\Q_p$ whose restriction to $K$ is $\sigma$. Let $\delta$ be a locally constant lift of $\eta\bar\theta^{-1}$, for example using the Teichm\"uller lift, then $\mathrm{Ind}_{\Gcal_{K'}}^{\Gcal_K}(\theta\delta)$ is a lift of $\bar r$ which has Hodge-Tate weights ${\bf k}$ and whose restriction to $\Gcal_{K'}$ is diagonal.
\end{proof}

We also assume that $p$ does not divide $2d$. Then Corollary $A.7$ of \cite{EmertonGee} and Proposition \ref{potdiag} tell us that we can find $F$ a totally real field, $E$ a totally imaginary quadratic extension of $F$ and a continuous irreducible representation $\bar\rho : \Gcal_F\rightarrow\Gcal_d(\bar\Fbb_p)$ (see for example \cite[\S$5.1$]{EmertonGee} for the definition of $\Gcal_d$) such that
\begin{itemize}
\item $4|[F:\Q]$ ;
\item each place $v|p$ of $F$ splits in $E$ and $F_v\simeq K$;
\item for each place $v|p$ of $F$, there is a place $\tilde v$ of $E$ dividing $p$ and such that $\bar\rho|_{\Gcal_{F_{\tilde v}}}\simeq \bar r$ ;
\item $\bar\rho$ is unramified outside of $p$ ;
\item $\bar\rho^{-1}(\GL_d(\bar\Fbb_p)\times\GL_1(\bar\Fbb_p))=\Gcal_E$ ;
\item $\bar\rho(\Gcal_{E(\zeta_p)})$ is adequate (in the sense of \cite[\S$2$]{Thorne})
\item $\bar\rho$ is automorphic, we will explain now what this means.
\end{itemize}

Let $v_1$ be a place\footnote{The introduction of this auxiliary place is only needed for the patching construction which will be used in the forthcoming lines} of $F$ which is prime with $p$ and satisfies the same hypothesis as in \cite[\S$5.3$]{EmertonGee}. We define the compact open subgroup $H=\prod_v H_v\subset G(\mathbb{A})$ so that $H_v\simeq\GL_d(\Ocal_K)$ if $v|p$, $H_v\subset G(F_v)$ is maximal hyperspecial if $v\nmid p$ and $v\neq v_1$, and $H_{v_1}$ an open pro-$\ell$-subgroup of $G(F_{v_1})$ for $\ell$ the residual characteristic of $v_1$. We say that $\bar\rho$ is automorphic if $S_W(H,L)_{\bar\rho}\neq0$ for some irreducible locally algebraic representation $W$ of $G(\mathcal{O}_F\otimes\Z_p)$.

{All constructions we did in section $3$ depend on the choice of a representation $W_\infty$ of the group $\prod_{v\in S'_p, w\in I_\infty(v)}G(F_w)$. Even if that's not explicit in the notation, the space $S(H^{v_0},L)$ actually depends on $W_\infty$. That's why we need the following lemma.}

\begin{lem}\label{existence}
There exists an irreducible algebraic representation $W_\infty$ of the group $\prod_{v\in S'_p}\prod_{w\in I_\infty(v)}G(F_v)$ such that $\hat S(H^{v_0},L)_{\bar\rho}\neq0$.
\end{lem}

\begin{proof}
By definition, we know that there exists an irreducible locally algebraic representation $W$ of $G(\mathcal{O}_F\otimes\Z_p)$ such that $S_W(H,L)_{\bar\rho}\neq0$. Let 
\[\hat S(H^p,L)_{\bar\rho}=\varprojlim_n\varinjlim_{H_p\subset G(\Ocal_F\otimes\Z_p)}S(H^pH_p,\Ocal_L/\varpi_L^n)_{\bar\rho}.\]
 We see, as in section \ref{densityautsection} that $S_W(H,L)_{\bar\rho}\simeq \Hom_{G(\Ocal_F\otimes\Z_p)}(W^*,\hat S(H^p,L)_{\bar\rho})$ and that, as a $G(\Ocal_F\otimes\Z_p)$-representation, $S(H^p,L)_{\bar\rho}$ is isomorphic to a non zero direct factor of $C(G(\Ocal_F\otimes\Z_p),L)^r$ for some $r\geq1$.\\
We can then find an irreducible algebraic representation $W^{v_0}$ of $\prod_{v\in S'_p}G(F_v)$ such that $\Hom_{\prod_{v\in S'_p}G(F_v)}(W^{v_0},\hat S(H^p,L)_{\bar\rho}\neq0$ and choose for $W_\infty$ the irreducible algebraic representation of $\prod_{v\in S'_p}\prod_{w\in I_\infty(v)}G(F_v)$ associated to $W^{v_0}$ as explained in section \ref{densityautsection}. Then we have $\hat S(H^{v_0},L)_{\bar\rho}\simeq\Hom_{\prod_{v\in S'_p}G(F_v)}(W^{v_0},\hat S(H^p,L)_{\bar\rho})$.
\end{proof}

%

\subsection{A result of density in crystalline deformation spaces}

Fix ${\bf n}$ a dominant algebraic weight such that $\Hom_{\GL_d(\Ocal_K)}(W_{\bf k}^*,\hat S(H^{v_0},L)_{\bar\rho})\neq0$. We write ${\bf k}$ for the strongly dominant weight associated to ${\bf n}$ by the recipe of $(\ref{fromntok})$.


{By Kisin's result \cite{Kisinpstdeform} there exists a reduced $p$-torsion free quotient $R_{\bar r, {\bf k}}^{\rm cris}$ of $R_{\bar r}$ such that a continuous $\zeta:\,R_{\bar r}\rightarrow\bar\Q_p$ factors through $R_{\bar r, {\bf k}}^{\rm cris}$ if and only if $\zeta\circ r^{\rm univ}$ is crystalline of Hodge-Tate weight ${\bf k}$. Here $r^{\rm univ}$ is the universal deformation of $\bar r$ on $R_{\bar r}$. 
Moreover the ring $R_{\bar r,\bf k}^{\rm cris}[1/p]$ is formally smooth. Let's denote by $\Xfrak_{\bar r,{\bf k}}^{\rm cris}$ the generic fiber of the formal scheme $\Spf R_{\bar r,\bf k}^{\rm cris}$. In this section we will use patching techniques to prove that the $R_{\bar r}$-closure (or equivalently the $R_{\bar r, {\bf k}}^{\rm cris}$-closure) of automorphic points of $\Xfrak_{\bar r,{\bf k}}^{\rm cris}$ is a union of connected components of $\Xfrak_{\bar r,{\bf k}}^{\rm cris}$. 


In our particular setting, we will say that a $\bar\Q_p$-point of $\Spec(R_{\bar r, {\bf k}}^{\rm cris})$ is automorphic if the corresponding deformation of $\bar r$ is isomorphic to the restriction to $\Gcal_{E_{w_0}}$ of a deformation $\rho$ of $\bar\rho$ associated to a an automorphic form of $G$ of type $W_\infty$. More precisely, let $\mathfrak{p}$ be the kernel of the map $R_{\bar\rho}[1/p]\rightarrow \bar\Q_p$ associated to $\rho$, then $\rho$ is called automorphic if $S(\GL_d(\Ocal_K)\tilde{H}^{v_0},L)[\mathfrak{p}]\neq0$ for some compact open subgroup $\tilde{H}^{v_0}\subset G(\Abb_F^{v_0})$.}

Further an irreducible component of $\Spec(R_{\bar r, {\bf k}}^{\rm cris}[1/p])$ is called automorphic if it contains an automorphic point.
\begin{rem}
Note that actually our definition of being automorphic depends on the representation $W_\infty$ chosen as in Lemma $\ref{existence}$. However, in order not to overload the notation we often will suppress $W_\infty$ from the notations. 
\end{rem}
Our goal in this section is to use the usual patching construction to prove the following result.

\begin{theo}\label{O+densecomp}
Let $X$ be an automorphic component of $\Spec(R_{\bar r, {\bf k}}^{\rm cris}[1/p])$. Then the set of automorphic points in $X$ is Zariski dense.
\end{theo}

Let $\Xfrak_{\bar r,{\bf k}}^{\rm aut}$ be the union the components of the rigid analytic generic fiber $\Xfrak_{\bar r,{\bf k}}^{\rm cris}$ of $\Spf R_{\bar r,{\bf k}}^{\rm cris}$ that correspond to automorphic components of $\Spec(R_{\bar r, {\bf k}}^{\rm cris}[1/p])$. 
\begin{cor}\label{densefibres}
The set automorphic points in $\Xfrak_{\bar r,{\bf k}}^{\rm aut}$ is $R_{\bar r}$-dense in $\Xfrak_{\bar r,{\bf k}}^{\rm aut}$.
\end{cor}

Given a rigid space $\Xfrak$ over $\Q_p$ we write $|\Xfrak|$ for the underlying point set of $\Xfrak$. Similarly we write $|X|$ for the set of closed points of a $\Q_p$-scheme $X$. 
Let $R$ be a complete local noetherian $\Z_p$-algebra with finite residue field and let $\Xfrak$ denote the generic fiber of $R$ in the sense of Berthelot. Further let $X=\Spec R[1/p]$. Then we have $|X|=|\Xfrak|$, as, if we write $R=\Z_p\dbl T_1,\dots, T_n\dbr/(f_1,\dots, f_m)$, then both sets are identified with the ${\rm Gal}(\bar\Q_p/\Q_p)$-orbits in 
\[\{x=(x_1,\dots, x_n)\in \bar\Q_p\mid |x_i|<1\ \text{and}\ f_j(x)=0\},\] compare also \cite[Lemma 7.1.9]{deJong}.
Further a subset $Z\subset |X|=|\Xfrak|$ is dense in $X$ if and only it is $R$-dense in $\Xfrak$. This proves that the theorem implies the corollary.

\begin{proof}[Proof of the Theorem]
Define $R_{v_1}^{{\square}}$ as the universal ring pro-representing the functor of lifts of $\bar\rho|_{G_{v_1}}$. {We have to take care to the fact that we fixed the weight $W_\infty=\bigotimes_{v\in S'_p,w\in I_\infty(v)}W_w$. As a consequence, for $v\in S'_p$, let $R_v'$ be the quotient $R_{\bar\rho|_{\Gcal_{E_{\tilde{v}}}},\bf k_v}^{\rm cris}$ of $R_{\bar\rho|_{\Gcal_{E_{\tilde{v}}}}}$ where ${\bf k}_v$ is the set of Hodge-Tate weights associated to the highest weight of the algebraic representation $\bigotimes_{w\in I_\infty (v)}W_w$. Let $R_{V_0}'=R_{\bar r,\bf k}$ and $R^{\rm loc}=R_{v_1}^\square \hat\bigotimes_{v\in S_p}R'_v$} and $R_{\infty,g}=R^{\rm loc}\dbl x_1,\dots,x_g\dbr$. Then the patching construction (see \cite[\S$6$]{Thorne}) gives us, for $g$ big enough, a $R_{\infty,g}$-module $M_{\infty}$ of finite type whose support in $\Spec R_{\infty,g}$ is a union of irreducible components and $M_\infty[1/p]$ is a projective $R_\infty[1/p]$-module. Let $R$ be the quotient of $R^{\rm loc}$ corresponding to one of these irreducible components. Then $R\otimes_{R^{\rm loc}}M_{\infty}[1/p]$ is a faithful $R\dbl x_1,\dots,x_g\dbr[1/p]$-module. Moreover, $M_{\infty}$ is constructed as an inverse limit of spaces $M_n$, where $M_n$ is a quotient of a space of automorphic forms, on which the action of $R_{\infty,g}$ factors through a ring of Hecke operators.\\
Lemma $3.4.12$ in \cite{Kisinmodularity} shows that the irreducible components of $\Spec R^{\rm loc}\dbl x_1,\dots, x_g\dbr$ are of the form $\Spec(\bar R)$ where $\bar R=\bigotimes_{v\in S_p\cup\{v_1\}} \bar R'_v$ and $\bar R'_v$ is an irreducible component of $R'_v$ if $v|p$ and of $R_{v_1}^\square$ if $v=v_1$. Now fix $\bar R'_{v_0}$ an automorphic irreducible component of $R_{\bar r, {\bf k}}^{\rm cris,\square}$. This means that there exists an irreducible component $\bar R'\dbl x_1,\dots,x_g\dbr$ of $R^{\rm loc}\dbl x_1,\dots, x_g\dbr$ containing $\bar R'_{v_0}$. If $t\in R_{\bar r, {\bf k}}^{\rm cris,\square}$ vanishes on each automorphic point, then $t$ acts trivially on the spaces $M_n$, and so on $M_{\infty}$. Now $R\otimes_{R^{\rm loc}}M_\infty[1/p]$ being a faithful $R[1/p]$-module, the image of $t$ in $R'_{v_0}[1/p]$ is zero which implies that $t=0$ in $R'_{v_0}$ since $R'_{v_0}$ is $p$-torsion free. 
\end{proof}
\subsection{Variation on the density of trianguline representations.}

In order to deduce density statements in the generic fiber of a local deformation ring from density statements in the space of trianguline representations one needs to show that the image of the space of trianguline representations is dense. This density result is included in the work of Chenevier \cite{Chtrianguline} and Nakamura \cite{Nakamura3}. In our case the situation will be a bit more restrictive: we only can make a statement about the components of the space of trianguline representations that are met by some eigenvariety. Hence we need to sharpen this density result a bit.
In order to do so we need to compare the fibers of the space of trianguline representations over strongly dominant algebraic weights with Kisin's crystalline deformation rings.

Let ${\bf k}=(k_{\sigma,i})$ be a strongly dominant weigh and let us continue to assume that $\bar r$ is absolutely irreducible.
In order not to overload the notation let us write $A_{\bf k}=R_{\bar r,{\bf k}}^{\rm cris}$ for the rest of this subsection.
Further (for the remaining of this subsection) we write
\begin{align*}
\Xfrak_{\bf k}&=(\Spf A_{\bf k})^{\rm rig}\\
X_{\bf k}&=\Spec (A_{\bf k}[1/p]).
\end{align*}  
\begin{defn}
We say that a crystalline representation $r_x:\Gcal_K\rightarrow \GL_d(\kappa(x))$ corresponding to a rigid analytic point $x\in \Xfrak_{\bf k}$ respectively to a closed point $x\in X_{\bf k}$ is \emph{regular} if the eigenvalues of the Frobenius on the Weil-Deligne representation $\WD(D_{\rm cris}(r_x))$ associated to $r_x$ are pairwise distinct. We further say that $x$ (resp.~$r_x$) is \emph{generic} if all $d!$ filtrations on $D_{\rm cris}(r_x)$ induced by an ordering of the Frobenius eigenvalues are in general position with respect to all Hodge-Tate filtrations (with respect to all embeddings $\psi:K\hookrightarrow L$). 
\end{defn}

\begin{rem}
Note that if $r_x$ is crystalline generic, then all parameters $(\delta_1,\dots,\delta_d)$ of all possible triangulations of its $(\varphi,\Gamma)$-module $D^\dagger_{\rm rig}(r_x)$ have the property that $(\delta_1|_{\Ocal_K^\times},\dots,\delta_d|_{\Ocal_K^\times})$ is algebraic of strongly dominant weight ${\bf k}$.
Moreover, given such a trianguline filtration $\Fil_i$ of $D^\dagger_{\rm rig}(r_x)$ the extensions
\[0\longrightarrow \Fil_i\longrightarrow \Fil_{i+1}\longrightarrow \Rcal_{k(x)}(\delta_i)\longrightarrow 0\]
 are non split. This may be checked after taking the quotient by $\Fil_{i-1}$ and hence we are reduced to the $2$-dimensional case where one easily checks that the direct sum of two $(\phi,\Gamma)$-modules $\Rcal(\delta)\oplus\Rcal(\delta')$ such that $\delta|_{\Ocal_K^\times}=\delta_{\Wcal}(a)\neq\delta_\Wcal(a')=\delta'|_{\Ocal_K^\times}$ is not generic: One of the two possible triangulations is not in general position with the Hodge filtration. 
It follows that $x$ is in the image of $X(\bar r)^{\rm reg}$ under the map $X(\bar r)\rightarrow\Xfrak_{\bar r}$.
\end{rem}

By \cite[Theorem {2.5.5}]{Kisinpstdeform} there is a (locally free) $A_{\bf k}\otimes_{\Z_p}K_0$ -module on $X_{\bf k}$ together with an ${\rm id}\otimes\varphi$-linear automorphism $\Phi:D\rightarrow D$ such that \[(D,\Phi)\otimes \kappa(x)\cong D_{\rm cris}(r_x)\] for all closed points $x\in X$. It follows that there is a Zariski open subset $X_{\bf k}^{\rm reg}\subset X_{\bf k}$ with closed points precisely given by the regular points. \\Especially there is a Zariski open subset $\Xfrak_{\bf k}^{\rm reg}\subset \Xfrak_{\bf k}$with rigid analytic points precisely given by the regular points in $\Xfrak_{\bf k}$.

Recall that we have a map $\omega:\Scal^{\rm ns}(\bar r)\rightarrow \Wcal^d$ to the weight space. We further view the weight ${\bf k}$ as an element of $\Wcal^d$. 
\begin{prop}\label{fibsvsKisin}
The map $\omega^{-1}({\bf k})\rightarrow \Xfrak_{\bar r}$ induces a map $g_{\bf k}:\omega^{-1}({\bf k})\rightarrow \Xfrak_{\bf k}$ which is \'etale over $\Xfrak_{\bf k}^{\rm reg}$. Further $g_{\bf k}^{-1}(\Xfrak^{\rm reg}_{\bf k})$ is open and dense in $\omega^{-1}({\bf k})$.\\
\end{prop} 
\begin{proof}
By Corollary $\ref{Zopeninfib}$ the map $\omega^{-1}({\bf k})\rightarrow \Xfrak_{\bar r}$ generically factors over $\Xfrak_{{\bf k}}$. As $\Xfrak_{{\bf k}}\subset \Xfrak_{\bar r}$ is Zariski-closed the first claim follows. It is further easy to see that the preimage of $\Xfrak_{\bf k}^{\rm reg}$ is open and dense. It remains to prove the claim on \'etaleness which, in the adic set up  can be checked using the infinitesimal lifting criterion, cf.~\cite[Definition 1.6.5]{Huber}. Consider an affinoid algebra $A$ with ideal $I\subset A$ satisfying $I^2=0$ and the diagram
\[\begin{xy}
\xymatrix{
\Sp(A/I)\ar[r] \ar[d]&\omega^{-1}({\bf k})\ar[d]\\
\Sp(A) \ar[r] &\Xfrak_{{\bf k}}^{\rm reg}.
}
\end{xy}\]
This diagram gives rise to a family of filtered $\phi$-modules over $A$ and we write $(D,\Phi_D)$ for the associated family of Weil-Deligne representations on $\Sp(A)$. Further the upper arrow in the diagram gives us a filtration 
\[0\subsetneq \overline\Fil^1\subsetneq \dots\subsetneq \overline\Fil^{d-1}\subset \overline\Fil^d=\overline D=D/I\]
by subspaces that are locally on $\Sp(A)$ direct summands and stable under the Frobenius $\Phi_{\overline D}=\Phi_D\mod I$ and we have to prove that this filtration uniquely lifts to a $\Phi_D$-stable filtration $\Fil^\bullet$ of $D$ such that locally on $\Sp(A)$ the $\Fil^i$ are direct summands of $D$.
After localizing on $\Sp(A)$ we may assume that $D$ is free and that all the $\overline{\Fil}^i$ are direct summands of $\overline D$. We show that we can lift $\overline \Fil^1$ uniquely to a $\Phi$-stable direct summand of $D$. The rest will follow by induction. 
Let us chose a basis $\bar e_1,\dots, \bar e_d$ of $\overline D$ such that $\overline \Fil^i$ is generated by $\bar e_1,\dots, \bar e_i$ and take arbitrary lifts $e_j$ of the $\bar e_j$ in $D$. We write $A=(a_{ij})$ for the matrix of $\Phi$ in this basis. Then the above implies that $a_{ij}\in I$ for $i>j$ and that $a_{ii}\neq a_{jj}$ for all $i\neq j$, as this is true modulo all maximal ideals of $A$ by definition of $\Xfrak_{\bf k}^{\rm reg}$. We have to show that there exists uniquely determined $\lambda_2\dots,\lambda_d\in I$ and a $\mu\in A^\times$ such that 
\[\Phi(e_1+\sum_{j=2}^d\lambda_j e_j)=\mu(e_1+\sum_{j=2}^d\lambda_j e_j).\]
However, this comes down to showing that 
\[(A-\mu E)e_1+\sum_{j=2}^d \lambda_j(A-\mu E)e_j=0\]
has (up to scalar) a unique solution with $\lambda_i\in I$ which is an easy consequence of $a_{ij}\in I$ for $i>j$ and $a_{ii}\neq a_{jj}$ for $i\neq j$.
\end{proof}

\begin{prop}\label{genericO+open}
There is a Zariski open and dense subset $X^{\rm gen}_{\bf k}\subset X_{\bf k}$ whose closed points are precisely the generic crystalline representations.  Especially there is an $R_{\bar r}$-open subset $\Xfrak_{\bf k}^{\rm gen}\subset \Xfrak_{\bf k}$ whose rigid analytic points are precisely the generic crystalline representations. 
\end{prop}
\begin{proof}
After twisting with some power of the cyclotomic character we may assume that all Hodge-Tate weights are non-negative. The ring $A_{\bf k}$ is a quotient of the universal deformation ring $R_{\bar r}$ and (with the notations of \cite[1.6.4]{Kisinpstdeform}) it is even a quotient of $R_{\bar r}^{\leq h}$ for some $h\gg 0$. By construction of the quotient $R_{\bar r}^{\leq h}$ there is a free ${A_{\bf k}\hat\otimes_{\Z_p}W}\dbl u\dbr$-module $\Mfrak$ together with an injection $\Phi:\Mfrak\rightarrow \Mfrak$ that is semi-linear with respect to the identity on $A$, the Frobenius on $W=\Ocal_{K_0}$ and $u\mapsto u^p$ and whose cokernel is killed by $E(u)^h$. Here $E(u)$ denotes the minimal polynomial of a chosen uniformizer of $K$ over $K_0$. 

By \cite[Theorem 2.2.5]{Kisinpstdeform} this module has the property that for all closed points $x\in X_{\bf k}$ with corresponding crystalline representation $r_x:\Gcal_K\rightarrow \kappa(x)$ one has a canonical isomorphism
\[D_{\rm cris}(r_x)\cong (\Mfrak/u\Mfrak,\Phi\,{\rm mod} \,u)\otimes_{A_{\bf k}} \kappa(x).\]
Let us write $\Nfrak=\Mfrak[1/p]$. By the construction in \cite[2.1, 2.2, (2.5.3) ]{Kisinpstdeform} resp.~\cite[5.a.1, (5.30)]{PappasRapo} we have an isomorphism
\[ \Nfrak/u\Nfrak\otimes_WK\cong \Phi(\phi^\ast\Nfrak)/E(u)\Phi(\phi^\ast \Nfrak)\]
under which (for all closed points $x\in X_{\bf k}$) the Hodge filtration on $D_{\rm cris}(r_x)\otimes_{K_0}K$ is induced by the filtration on $ \Phi(\phi^\ast\Nfrak)/E(u)\Phi(\phi^\ast \Nfrak)$ by the $E(u)^i\Nfrak/E(u)\Phi(\phi^\ast \Nfrak)$. As the condition for two possible filtrations to be in general position is an open condition it follows that there is a Zariski open subset $X_{\bf k}^{\rm gen}\subset X_{\bf k}$ whose closed points are precisely the generic crystalline points. 

It remains to show that every connected component of $X_{\bf k}$ contains a generic point (recall that $X_{\bf k}$ is formally smooth and hence every connected component is irreducible). However, this is proven exactly as in \cite[Proposition 4.3]{Chtrianguline}.
\end{proof}

We now return to the setup with eigenvarieties. The data of $F$, $E$, $G$, $H$, $\bar\rho$ are checking the properties of section \ref{eigenvarieties}, so that we can consider the eigenvariety $Y_{\bar\rho}=Y(W_\infty,S,e)_{\bar\rho}$ with $S=S_p\cup\{v_1\}$ and $e$ a well suited idempotent.
\begin{defn}
 Write $X(\bar\rho,W_\infty)^{\rm aut}$ for the union of irreducible components of $X(\bar r)$ whose intersection with the regular part $X(\bar r)^{\rm reg}\subset X(\bar r)$ is met by any of the eigenvarieties $Y(W_\infty,S,e)_{\bar\rho}$ under the map defined in Theorem $\ref{maptofinslopespace}$ for a some $(S,e)$.
 \end{defn}
Using Proposition \ref{genericO+open} and Corollary \ref{densefibres}, one can see that $X(\bar\rho,W_\infty)^{\rm aut}$ can be characterized as the union of irreducible components of the space $X(\bar r)$ whose intersection with the regular part $X(\bar r)^{\rm reg}$ contains a crystalline automorphic point of strongly dominant Hodge-Tate weights (note that in our context the assumption that the Hodge-Tate weights are strongly dominant is a consequence of the automorphy). Note that, as a consequence of Lemma \ref{existence}, Theorem \ref{O+densecomp}, Proposition \ref{genericO+open} and the fact that generic crystalline points lie in $X(\bar r)^{\rm reg}$, the space $X(\bar\rho,W_\infty)^{\rm aut}$ is non empty.



\begin{theo}\label{crystrepdense}
The image of $X(\bar\rho,W_\infty)^{\rm aut}$ in $\Xfrak_{\bar\rho_{w_0}}=\Xfrak_{\bar r}$ is Zariski-dense in a union of irreducible components of $\Xfrak_{\bar\rho_{w_0}}$. 
\end{theo}
\begin{proof}
Let us write $X$ for the Zariski-closure of the image of $X(\bar\rho,W_\infty)^{\rm aut}$ for the moment. Following the proof of \cite[4.5]{Chtrianguline} and \cite[Theorem. 4.3]{Nakamura3}, we are reduced to show that there exists a generic crystalline point $r \in X^{\rm sm}$ with the following additional properties
\begin{enumerate}
\item[(i)] if $\lambda_1,\dots,\lambda_d$ are the (pairwise distinct) eigenvalues of the Frobenius on $\WD(D_{\rm cris}(r))$, then $\lambda_i\lambda_j^{-1}\notin p^\Z$ for all $i,j$.
\item[(ii)] the space $X(\bar\rho,W_\infty,e)^{\rm aut}$ contains all possible triangulations of the representation $r$.
\end{enumerate}
Here $X^{\rm sm}\subset X$ is the smooth locus which is Zariski-open and dense in $X$.

Let us write $\pi:X(\bar\rho,W_\infty)^{\rm aut}\rightarrow \Xfrak_{\bar r}$ for the canonical projection that by definition factors over $X$. 

Fix a connected component $\mathcal{C}\subset X(\bar \rho,W_\infty)^{\rm aut}\cap X(\bar r)^{\rm reg}$ such that $\pi^{-1}(X^{\rm sm})\cap \mathcal{C}\neq \emptyset$ and hence $\pi^{-1}(X^{\rm sm})\cap\Ccal$ is Zariski-open and dense in $\Ccal$. Then, by definition of $X(\bar\rho,W_\infty)^{\rm aut}$, there is a pair $(S,e)$ such that $Y_{\bar\rho}=Y(W_\infty,S,e)_{\bar\rho}\neq\emptyset$ and a point $y_0\in Y_{\bar \rho}$ such that $f(y_0)\in\Ccal$. Here $f:Y_{\bar\rho}\rightarrow X(\bar\rho,W_\infty)^{\rm aut}$ is the map defined by Theorem $\ref{maptofinslopespace}$.

We claim that there is classical point $y_1\in Y_{\bar\rho}$ such that its image $r_{y_1}$ in $\Xfrak_{\bar r}$ lies on an automorphic component $C$ of $\Xfrak_{{\bf k}}\subset \Xfrak_{\bar r}$ for some strongly dominant weight ${\bf k}=\omega_Y(y_1)$ such that $C\cap X^{\rm sm}$ is non-empty (and hence Zariski-open and dense in $C$).
Indeed it suffices to construct $y_1$ such that $r_{y_1}\in C$ and $C\cap X^{\rm sm}$ is non-empty: as $y$ is associated with an automorphic representation the component $C\subset \Xfrak_{{\bf k}}$ is automatically automorphic.

Let $U\subset X(\bar r)^{\rm reg}\cap X(\bar\rho,W_\infty)^{\rm aut}$ be a connected open neighborhood of a chosen point $x_0\in X(\bar r)^{\rm reg}\cap X(\bar\rho,W_\infty)^{\rm aut}$ {such that $\omega_d|_U$ has connected fibers. Indeed such a neighborhood exists by Lemma $\ref{locallyaproduct}$.
As $\pi^{-1}(X^{\rm sm})\cap \Ccal$ is Zariski-open and dense in $\Ccal$, we find that $\pi^{-1}(X^{\rm sm})\cap U$ is non-empty and hence it is Zariski-open and dense in $U$.  

Let $V\subset Y_{\bar\rho}$ be} a quasi-compact neighborhood of $y_0$ such that $f(V)\subset U$. 
Then there are constants  $M_1,\dots, M_d$ such that
\[M_i\geq {\rm val}_x(\psi_{v_0}(t_{v_0,1}\dots t_{v_0,i})(x))+1\]
for all $x\in V$, where ${\rm val}_x$ is the valuation on $k(x)$ normalized by ${\rm val}_x(p)=1$.
As the strongly dominant algebraic weights ${\bf k}\in\Wcal^d$ such that  $M_i< k_{\sigma,i}-k_{\sigma,i+1}$
accumulate at $\omega(y_0)$ there is a weight ${\bf k}$ such that 
\begin{align*}
 k_{\sigma,i}-k_{\sigma,i+1}&>M_i,\\ 
\omega_d^{-1}({\bf k})\cap {U} \cap \pi^{-1}(X^{\rm sm})&\neq \emptyset\ \text{and}\\ 
\omega_d^{-1}({\bf k})\cap {U}\cap f(Y_{\bar\rho})&\neq \emptyset.
\end{align*} 

It follows that a point $y\in {V}$ with $f(y)\in \omega^{-1}({\bf k})\cap U$ is classical (i.e.~associated to an automorphic representation) and has the required property that $X^{\rm sm}\cap C\neq \emptyset$, where $C\subset \Xfrak_{\bf k}$ is the connected component containing $r_y=\pi(f(y))$: 
as $U\cap \omega_d^{-1}(\bf k)$ is connected it maps to $C$ and hence $\omega_d^{-1}({\bf k})\cap U \cap \pi^{-1}(X^{\rm sm})\neq \emptyset$ implies $X^{\rm sm}\cap C\neq \emptyset$


By Corollary $\ref{densefibres}$ the automorphic points are $R_{\bar r}$-dense in $C$ and by Proposition $\ref{genericO+open}$ the generic locus is $R_{\bar r}$-open in $C$. Hence there exists {a pair $(S',e')$ and a point $y'\in Y'_{\bar\rho}=Y(W_\infty,S',e')_{\bar\rho}$ (up to a lowering of the tame level $H^{v_0}$) such that the crystalline representation $r_{y'}$ given by the image of $y'$ in $\Xfrak_{\bar r}$ is generic and $y'$ lies on $C\subset \Xfrak_{\bf k}\subset \Xfrak_{\bar r}$.}
We write $f':Y'_{\bar\rho}\rightarrow X(\bar r)$ for the map defined by Theorem $\ref{maptofinslopespace}$ corresponding to the eigenvariety $Y'_{\bar\rho}$.

Now the triangulations of $r_{y'}$ are in bijection with the orderings of the Frobenius eigenvalues as the representation is regular and hence there are exactly $d!$ such triangulations. 
We claim that all these possible triangulations of $r_{y'}$ lie in ${f'(Y'_{\bar\rho})}$. \\
Let us write $\Pi$ for the automorphic representation associated to $y'$. 
As $\Pi_{v_0}$ is unramified, there are $d!$ distinct characters $\chi_i$ of $T_{v_0}/T_{v_0}^0$ such that $\Pi_{v_0}{|\det|^{\frac{1-d}{2}}}$ appears in the parabolic induction ${\rm Ind}_B^{\GL_d(K)}\chi_i$ and as the representation is generic the Hodge filtrations is in general position with the triangulation given by the ordering of the Frobenius eigenvalues.  It follows that the parameters of the triangulation are prescribed by the character $\chi_i$.

This shows that there are points $z_1,\dots, z_{d!}\in {X(\bar\rho,W_\infty)}^{\rm aut}$ with $z_i={f'}(\Pi,\chi_i)$ mapping to $r_{y'}\in\Xfrak_{\bar r}$ and $z_1,\dots,z_{d!}$ are precisely the $d!$ possible triangulations of the crystalline representation $r_{y'}$. Note that the assumption that $r_{y'}$ is generic implies that the trianguline filtrations have to be nowhere split and hence $z_i\in X(\bar\rho,W_\infty)^{\rm aut}\cap X(\bar r)^{\rm reg}$.
Using Proposition $\ref{Xreglarge}$ we may identify open subsets of $X(\bar\rho)^{\rm reg}$ with open subsets of $\Scal^{\rm ns}(\bar\rho)$.

Now there exist connected open neighborhoods $U_i\subset  X^{\rm reg}(\bar r)\cap X(\bar\rho,W_\infty)^{\rm aut}\cap \omega^{-1}({\bf k}) $ of the $z_i$ such that $U_i\cap U_j=\emptyset$. As the map $g_{{\bf k}}$ from Proposition $\ref{fibsvsKisin}$ is \'etale  at all the $z_i$ it follows from \cite[Proposition 1.7.8]{Huber} that $g_{\bf k}$ is open in some neighborhood of the $z_i$.  Hence it follows that (after eventually shrinking the $U_i$) the image of the $U_i$ in $\Xfrak_{\bf k}^{\rm reg}\cap C$ is open and after shrinking the $U_i$ even further we may assume that $U_z=g_{{\bf k}}(U_i)=g_{{\bf k}}(U_j)$ for all $i,j$. It follows that all the crystalline representations $r \in U_z$ have the property that $X(\bar\rho,W_\infty)^{\rm aut}\cap X^{\rm reg}(\bar r)$ contains all their possible triangulations. 

Finally the subset of all {generic} points $r\in C$ such that  the Frobenius eigenvalues satisfy condition (i) from above is a Zariski-open and dense subset of $C'\subset C$. Hence $U_z\cap C'\cap X^{\rm sm}\neq \emptyset$  and an element of this intersection is a point lying in $X^{\rm sm}$ and checking the conditions (i) and (ii) from above.
\end{proof}

\subsection{End of the proof}

Let us keep the notations from the preceding subsection. Fix a strongly dominant weight\footnote{not necessarily as in preceding section} ${\bf k}=(k_{\sigma,i})_\sigma\in \prod_\sigma\Z^d$ and recall the subset $\Wcal^d_{{\bf k}, \rm la}\subset \Wcal^d$. Recall further that we wrote $X_{\bf k}(\bar r)\subset X(\bar r)$ for the $R_{\bar r}$-closure of all crystabelline points of Hodge-Tate weight ${\bf k}$.

In this section we prove the following theorem which will imply the desired result on the density of potentially crystalline representations of fixed weight. 


%
%

\begin{theo}\label{densityinfinslopespace}
We have an inclusion $X(\bar\rho,W_\infty)^{\rm aut}\subset X_{\bf k}(\bar r)$.
\end{theo}
\begin{proof}
It follows from Corollary $\ref{Rdensityoneigenvar}$ (resp. Theorem $\ref{ImageEV}$) that 
\[f(Y_{\bar\rho})\subset X_{\bf k}(\bar r)\] 
for every eigenvariety $Y_{\bar\rho}=Y(W_\infty,S,e)_{\bar\rho}$.

%
Let $(S,e)$ and $Y_{\bar\rho}=Y(W_\infty,S,e)_{\bar\rho}$ be chosen so that we can find a point $y\in Y_{\bar\rho}$ such that $x=f(y)\in X(\bar r)^{\rm reg}$. By definition $\omega(x)=\omega_Y(y)$ is a strongly dominant algebraic weight ${\bf k}'$. Let us write $X_{\bf k'}$ for the intersection of $\omega^{-1}({\bf k}')$ with $X\cap X(\bar r)^{\rm reg}$. Let $C$ denote the connected component of $X_{\bf k'}$ containing $x$ and let $C^{\rm cris}\subset C$ be the Zariski-open (and dense) subset of crystalline points. 

By construction $C$ maps under the projection to $\Xfrak_{\bar r}$ to a connected component $C'$ of $\Xfrak_{\bar r,{\bf k}'}^{\rm cris}$. We write $R'$ for the quotient of $R_{\bar r,{\bf k}}^{\rm cris}$ corresponding to $C'$.
Further the component $C'$ is an automorphic component, as by assumption the representation defined by $x$ extends to an automorphic Galois representation. 

Let $t\in R_{\bar r}$ be an element vanishing on $X_{\bf k}(\bar r)$ and consider its image, still denoted by $t$, in $R'$. Let $z'\in C'$ be an automorphic point corresponding to a crystalline $\Gcal_K$-representation $r_{z'}$. By definition, we can find an irreducible cuspidal automorphic representation $\Pi$ of an unitary group $G$, of some level $\tilde{H}\subset H$ as in section \ref{eigenvarieties} such that $(\rho_{\Pi})_{w_0}\simeq r_{z'}$. Then there exists a triple $(W_\infty,S',e')$ and a character $\chi$ of $T/T^0$ such that $(\Pi,\chi)\in\Zcal$. Here $\Zcal$ is the set of classical points of the eigenvariety $Y(W_\infty,S',e')$. Now Theorem $\ref{ImageEV}$ implies that the image of $(\Pi,\chi)\in Y(W_\infty,S',e')_{\bar\rho_{\Pi}}$ in $X(\bar r)$ is in fact contained in $X_{\bf k}(\bar r)$, {or equivalently that $t(z')=0$}. {From} Corollary \ref{densefibres} {we can deduce that} the image $t$ in $R'$ is zero. As the image of $C$ under $\pi:X(\bar r)\rightarrow \Xfrak_{\bar r}$ {is contained in} $C'$, the elements $t$ vanishes on $C$, which implies that $C\subset X(\bar\rho,W_\infty)^{\rm aut}$.

We can now conclude. Fix $\Ccal$ a component of $X(\bar\rho,W_\infty)^{\rm aut}$ and $y_0\in Y_{\bar\rho}$ such that $f(y_0)=x_0\in X\cap X(\bar r)^{\rm reg}$. Let ${\bf k}_0'=\omega_Y(y)$ and pick $\tilde{x}_0\in\Scal^{\rm ns}(\bar r)$ such that $\pi_{\bar r}(\tilde{x}_0)=x_0$. By Lemma $\ref{locallyaproduct}$ there is a quasi-compact connected neighborhood $U$ of $x$ inside $X^{\rm reg}(\bar r)$ such that $U$ is isomorphic to a product of an open subset $U_1\subset \Wcal^d$ with a rigid space $U_2$ which we may chose to be connected. After shrinking $U_1$ we may also assume that $f(Y_{\bar\rho})\cap U$ surjects onto $U_1$. 
As $U$ is quasi-compact, there exist $M_1,\dots, M_d$ such that
\[M_i\geq {\rm val}_x(\psi_{v_0}(t_{v_0,1}\dots t_{v_0,i})(x))+1\]
for all $x\in U$, where ${\rm val}_x$ is the valuation on $k(x)$ normalized by ${\rm val}_x(p)=1$. Let us write $Z_1\subset U_1$ for the set of strongly dominant algebraic weights $k_{\sigma,1}> \dots> k_{\sigma,d}$ such that $M_i< k_{\sigma,i}-k_{\sigma,i+1}$ for all $i$ and $\sigma:K\hookrightarrow \bar\Q_p$.
Then $Z_1$ is Zariski dense in $U_1$ and hence $\omega^{-1}(Z_1)\cap U$ is Zariski dense in $U$.

A point  $y\in f^{-1}(U)$ mapping to ${\bf k'}\in Z_1$ is classical by the choice of $M_i$ and \cite[Theorem 1.6 (vi)]{Ch3}. Applying what preceeds with the point $y$ and the component of $X_{\bf k'}$ containing $\omega^{-1}(\{{\bf k'}\})\cap U\simeq U_2$ we find that $\omega^{-1}(\{{\bf k'}\})\cap U\subset X_{\bf k}(\bar r)$. This implies that $X_{\bf k}(\bar r)$ contains $\omega^{-1}(Z_1)\cap U$ which is Zariski dense in $U$ and hence $U\subset X_{\bf k}(\bar r)$. This implies that $X_{\bf k}(\bar r)$ contains $\Ccal$.


%
\end{proof}

\begin{theo}
Let $p\nmid 2d$ and let $K$ be a finite extension of $\Q_p$. Let  $\bar r: \Gcal_K\rightarrow \GL_d(\Fbb)$ be an absolutely irreducible continuous representation and let $R_{\bar r}$ be its universal deformation ring. 
Assume that $\bar r\not\cong \bar r(1)$. Let ${\bf k}=(k_{i,\sigma})\in \prod_{\sigma:K\hookrightarrow \bar\Q_p}\Z^d$ be a strongly dominant weight. Then the representations that are crystabelline of labeled Hodge-Tate weight ${\bf k}$ are Zariski-dense in $\Spec R_{\bar r}[1/p]$.
\end{theo}
\begin{proof}


The assumptions that $\bar r$ is absolutely irreducible and $\bar r\not\cong \bar r(1)$ imply that $Z=\Spec R_{\bar r}[1/p]$ is smooth and irreducible.

Let $X=X(\bar\rho,W_\infty)^{\rm aut}\subset X(\bar r)$ for a suitable choice of $W_\infty$ as in Lemma $\ref{existence}$. 
Our assumptions imply that the $X$ is non-empty and hence Theorem $\ref{crystrepdense}$ implies that $X$ has dense image in $Z$.  
Let $t\in R_{\bar\rho}$ be a function vanishing on all crystabelline points of weight ${\bf k}$. Then Corollary $\ref{Zopeninfib}$ implies that it vanishes on $X\cap \omega^{-1}(\Wcal^d_{{\bf k},{\rm la}})$ and hence by Theorem $\ref{densityinfinslopespace}$ is vanishes on $X$. The claim follows as $X$ has dense image in $Z$.
\end{proof}

%

\bigskip

\parbox[t]{8.2cm}{ 
Eugen Hellmann\\
Mathematisches Institut\\
Universit\"at Bonn\\
Endenicher Allee 60\\
D-53115 Bonn \\
Germany
\\[1mm]
hellmann@math.uni-bonn.de
}
\parbox[t]{8.2cm}{ 
Benjamin Schraen \\ 
Laboratoire de Math\'ematique\\
Universit\'e de Versailles St.~Quentin\\
45 Avenue des \'Etats Unis\\
F-78035 Versailles
\\ France
\\[1mm]
{ benjamin.schraen@uvsq.fr}
}

\end{document}